\documentclass[a4paper,12pt,twoside,reqno]{amsart}

\usepackage[a4paper,margin=1.15in]{geometry}
\usepackage{amsmath,amssymb,amsthm,mathtools}
\usepackage{booktabs,tabularx}
\numberwithin{equation}{section}
\usepackage[utf8]{inputenc}
\usepackage{mathrsfs}
\usepackage{eucal} 
\usepackage{microtype}
\usepackage{comment}
\usepackage{enumitem}

\makeatletter
\def\cr@pref#1:#2\@nil{#1}
\@namedef{cr@map@lem}{Lemma}        \@namedef{cr@map@lem@p}{Lemmas}
\@namedef{cr@map@thm}{Theorem}      \@namedef{cr@map@thm@p}{Theorems}
\@namedef{cr@map@thmAlph}{Theorem}  \@namedef{cr@map@thmAlph@p}{Theorems}
\@namedef{cr@map@prop}{Proposition} \@namedef{cr@map@prop@p}{Propositions}
\@namedef{cr@map@cor}{Corollary}    \@namedef{cr@map@cor@p}{Corollaries}
\@namedef{cr@map@rem}{Remark}       \@namedef{cr@map@rem@p}{Remarks}
\@namedef{cr@map@ex}{Example}       \@namedef{cr@map@ex@p}{Examples}
\@namedef{cr@map@sec}{Section}      \@namedef{cr@map@sec@p}{Sections}
\@namedef{cr@map@tab}{Table}        \@namedef{cr@map@tab@p}{Tables}
\newcommand{\cr@tsing}[1]{\@nameuse{cr@map@\cr@pref#1:\@nil}}
\newcommand{\Cref}[1]{\cr@tsing{#1}~\ref{#1}}
\makeatother

\theoremstyle{plain}
\newtheorem{theorem}{Theorem}[section]
\newtheorem{lemma}[theorem]{Lemma}
\newtheorem{proposition}[theorem]{Proposition}
\newtheorem{corollary}[theorem]{Corollary}
\newtheorem{thmAlph}{Theorem}

\theoremstyle{definition}
\newtheorem{remark}[theorem]{Remark}
\newtheorem{example}[theorem]{Example}

\DeclareMathOperator{\Pic}{Pic}
\DeclareMathOperator{\rank}{rank}

\DeclareMathOperator{\Tor}{Tor}
\DeclareMathOperator{\pr}{pr}

\newcommand{\PP}{\mathbf{P}}
\newcommand{\CC}{\mathbf{C}}
\newcommand{\QQ}{\mathbf{Q}}
\newcommand{\ZZ}{\mathbf{Z}}
\newcommand{\Oc}{\mathscr{O}}
\newcommand{\Om}{\Omega^1}
\newcommand{\sm}{\mathrm{sm}}
\newcommand{\Pdual}{\check{\PP}^2}
\newcommand{\Rc}{\mathcal{R}}

\newcommand{\Mmod}{\mathscr{M}}
\newcommand{\Jf}{J_{g}}
\usepackage[dvipsnames]{xcolor}
\usepackage{hyperref}
\hypersetup{
    colorlinks=true,
    linkcolor=NavyBlue,
    citecolor=TealBlue,
    filecolor=NavyBlue,
    urlcolor=magenta
}
\usepackage{bookmark}

\usepackage{tikz-cd}
\usepackage{soul}
\usepackage[all]{xy}

\def\ra{\rightarrow}

\newcommand{\tsk}[1]{\textcolor{YellowOrange}}

\newcommand{\tsc}{T_S|_C}

\newcommand{\codim}{\mathrm{codim}}
\newcommand{\HH}{\mathrm{H}}
\newcommand{\OO}{\mathscr{O}}
\theoremstyle{plain}

\theoremstyle{definition}

\begin{document}

\title[Ramification of the moduli map]{Ramification of the moduli map for hyperplane sections of K3s and hypersurfaces}
\author{Dario Faro}
\address{Dipartimento di Matematica ``Federigo Enriques'', Universit\`a degli Studi di Milano, Via Saldini 50, 20133 Milano, Italy}
\email{dario.faro@unimi.it}
\author{Frank Gounelas}
\address{Mathematisches Institut, Universit\"at Bonn, Endenicher Allee 60, 53115 Bonn, Germany} 
\email{gounelas@math.uni-bonn.de}
\date{\today}
\subjclass[2020]{Primary 14J28, 14H10; Secondary 14J60, 14J70}
\keywords{K3 surfaces, hypersurfaces, moduli maps, ramification, logarithmic tangent bundles, jumping lines}

\begin{abstract}
Finiteness and injectivity of the moduli morphisms
\[
\mu_n\colon|nH|_{\sm}\longrightarrow\Mmod_{g_n}
\]
associated with linear systems on K3 surfaces are understood in many cases.
In this paper we study their ramification. Existing vanishing and
stability results already imply unramifiedness when the polarisation or the multiple
is sufficiently positive, so our focus is on low-degree phenomena.

In particular, for a general K3 surface of degree $2$, we prove that the primitive moduli map is quasi-finite but ramified at exactly $171$ points, identified with the jumping lines of a logarithmic tangent bundle associated with the branch sextic. On the other hand, we prove that if $X\subset\PP^n$ is a general hypersurface of degree at least $4$, then $\HH^0(Y,T_X|_Y)=0$ for every smooth hyperplane section $Y$ of $X$. In particular, the moduli map for smooth hyperplane sections of a general quartic K3 surface is unramified. 
We moreover show that ramification does occur for higher multiples of the primitive polarisation on both degree-$2$ and degree-$4$ K3 surfaces, and relate the question to the stability of the tangent bundle.
\end{abstract}

\maketitle

\tableofcontents

\section{Introduction}

Let $S$ be a complex projective K3 surface with $H$ an ample indivisible (in the Picard group) line bundle with $H^2=2d=2g-2$. The smooth members of $|H|=\PP^g$ are integral curves of genus $g$. For each $n\ge1$, set $g_n=1+n^2(g-1)$. We consider three questions about the moduli map
\[
\begin{tikzcd}
{|nH|_{\sm}} \arrow[r, "\mu_n"] & \Mmod_{g_n}
\end{tikzcd}
\]
taking a curve to its isomorphism class in the coarse moduli space of smooth curves of genus $g_n$.
\begin{description}[style=sameline,leftmargin=2em,labelsep=.5em,
    itemsep=.4\baselineskip,topsep=.5\baselineskip]
\item[\textbf{(Finiteness)}] Is $\mu_n$ quasi-finite onto its image?
\item[\textbf{(Injectivity)}] Is $\mu_n$ injective?
\item[\textbf{(Unramifiedness)}] Is $\mu_n$ unramified at every point of $|nH|_{\sm}$?
\end{description}

This morphism is known to be generically finite onto its image (\cite[Corollary 2]{BA}, \cite[Proposition 5.2]{DH}). Moreover, if $g \geq 11$ and the Clifford index of a curve in $|H|$ is at least $3$, the map is quasi-finite (\cite[Corollary 8.6]{CDS}). In recent work of Chen--Gounelas \cite{CG}, $\mu$ is shown to be quasi-finite for the primitive polarisation of any non-uniruled surface of Picard rank one, in particular covering the low genus cases also. Results on injectivity are given in \cite{FeyzbakhshI,FeyzbakhshII,CLW}. We give a more detailed summary of the known finiteness, unramifiedness and injectivity ranges in \Cref{sec:high-degree}.

Because $K_S=\Oc_S$, adjunction gives $N_{C/S}=\Oc_C(H)=\omega_C$, and the normal-bundle sequence identifies the kernel of the differential of $\mu$
\[     
   \ker\, d\mu_{[C]} \cong \HH^0(C,T_S|_C)
\] 
(see \Cref{lem:kernel}). Thus $\HH^0(C,T_S|_C)$ \emph{is} the obstruction to $\mu$ being unramified at $[C]$. The natural question of whether it vanishes for \emph{every} smooth $C$ on a (very) general $(S,H)$ asks whether $\mu_n$ is unramified everywhere. This problem is intimately related to the stability of the restriction $T_S|_C$ and the local isotriviality of the linear system. It is classical that for $H^2\gg0$ the restriction $T_S|_C$ is stable of slope $0$, hence has no sections, so the question is what happens for small values of $H^2$.

Our first main result (in \Cref{thm:very-general_h2}) settles the lowest degree case $H^2=2$, where the K3 surface is a double cover of $\PP^2$ branched along a general sextic curve $B$. Here we find that the moduli map is in fact ramified.

\begin{thmAlph}\label{thm:A}
For a general degree-two K3 surface, the moduli map $\mu_1$ is quasi-finite and ramified at exactly $171$ points. These correspond to the $171$ jumping lines of the logarithmic tangent bundle $T_{\PP^2}(-\log B)$ for $B$ a general plane sextic curve.
\end{thmAlph}

Our second main result (\Cref{thm:vanishing}) addresses the case of higher degree hypersurfaces, starting with degree-four K3 surfaces (quartic surfaces in $\PP^3$):

\begin{thmAlph}\label{thm:B}
Let $X \subset \PP^n$ be a general smooth hypersurface, where $n\ge3$, of degree $d\ge4$. Then for every smooth hyperplane section $Y=X\cap H$, one has $\HH^0(Y,T_X|_Y)=0$.
\end{thmAlph}
It implies that for a general degree-$4$ K3 surface the moduli map $\mu_1$ is unramified everywhere; see \Cref{cor:quartic-unramified}. 
More generally, the preceding theorem implies that the moduli map sending
a smooth hyperplane section of $X$ to its isomorphism class is unramified everywhere on $|\mathcal O_X(1)|_{\mathrm{sm}}$ (see
\Cref{remarkunramifiedness}).

The known finiteness, ramification and injectivity results are summarised in
\Cref{tab:moduli-map-summary} at the end of \Cref{sec:high-degree}.

The paper is organised as follows. In \Cref{sec:high-degree} we recall the description of the kernel of the differential of the moduli map and collect the relevant vanishing, stability and reconstruction results, including a summary of the known quasi-finiteness, unramifiedness and injectivity ranges.  In \Cref{sec:hypersurfaces} we recall the translation of  the infinitesimal problem for hyperplane sections of a hypersurface into a multiplication problem in its Jacobian ring (following \cite{BE}). This, together with an incidence estimate then proves the vanishing theorem for every smooth hyperplane section of a general hypersurface. 

The main part of \Cref{sec:low-degree-k3} is devoted to degree-$2$ K3 surfaces. Such a surface is a double cover of $\PP^2$ branched over a smooth sextic $B$ and we identify the ramification points of the moduli map with the transverse jumping lines of the logarithmic tangent bundle $T_{\PP^2}(-\log B)$. After that  we construct an explicit sextic with finite jumping scheme, whose jumping-line calculation is verified computationally in Appendix~\ref{app:jumping}. Then,  a determinantal argument shows that, for a general sextic, this scheme consists of exactly $171$ reduced points (see also Remark~\ref{Thom-Porteous} for a computation using the Thom--Porteous formula).  This proves \Cref{thm:A}. We then treat degree $4$, where the hypersurface vanishing theorem of \Cref{sec:hypersurfaces} implies that the moduli map for a general quartic K3 surface is unramified at every smooth hyperplane section. In both degrees $2$ and $4$ we also give geometric arguments towards quasi-finiteness, already known by other methods, using Hassett's results on stable reductions of $A_k$-singularities and Aluffi--Faber's results on limits of orbits of plane curves under the action of $\mathrm{PGL}_3$.

In \Cref{sec:stability-ramification} we compare ramification with the stability of $T_S|_C$ via various explicit examples. The examples show that instability need not cause ramification. We also give examples in degrees $2$ and $4$ for which the higher-multiple maps $\mu_n$ do have positive-dimensional ramification loci. In \Cref{section:jumpingloci} we place these phenomena in families by introducing the cohomology jumping loci $S_k(H)$, giving them a determinantal description and deriving codimension bounds.

\medskip
\noindent\textbf{Acknowledgements.} We thank Y. Dutta, C. Faber, S. Feyzbakhsh, D. Huybrechts, A. Kuznetsov and E. Sernesi for helpful conversations.
The first author would like to thank the Max Planck Institute for Mathematics in Bonn, where this work started. The second author would like to thank the ERC Synergy Grant HyperK (ID 854361) for its support. 

The LLMs OpenAI GPT-5.6-Sol and Claude Fable were used to find the examples in \Cref{prop:reduced-example_h2}, \Cref{rem:nonreduced-example_h2} and to write \Cref{prop:strata-two-parts} as well as to arrange and spell-check this text. 

\medskip
\noindent\textbf{Notation.} 
We work over the field of complex numbers throughout. For
$\mu_n\colon |nH|_{\sm} \rightarrow \Mmod_{g_n}$ the moduli map of some linear system, we often just denote $\mu_1$ by $\mu$. For a K3 surface $S$, we often freely use the
isomorphism $T_S\cong\Om_S$.

\section{Background results}\label{sec:high-degree}

Let $S$ be a complex projective K3 surface and let $H$ be an ample
primitive line bundle with $H^2=2g-2$.  In this section we will recall the known results concerning the finiteness, injectivity, and ramification of the map
\[
\mu\colon |H|_{sm} \rightarrow \Mmod_g.
\] Some of the results below hold
for arbitrary $(S,H)$, while others require $\Pic(S)=\ZZ[H]$ or a
general point of the relevant moduli space; these hypotheses will
always be stated explicitly. 
As already mentioned in the introduction, the main object we are interested in this paper is the differential of $\mu$. The following basic lemma is  well known.

\begin{lemma}\label{lem:kernel}
Let $S$ be a K3 surface, let $H\in\Pic(S)$, and let $C\in|H|$ be a
smooth curve of genus at least $2$. Then there is a canonical
isomorphism $\HH^0(C,T_S|_C)\cong\ker d\mu_{[C]}$; moreover, writing
$s_C\in \HH^0(S,\Oc_S(H))$ for an equation of $C$, multiplication by $s_C$ gives
\[
\begin{tikzcd}
\HH^1(S,T_S(-H)) \arrow[r, "\cup\,s_C"]
& \HH^1(S,T_S).
\end{tikzcd}
\]
Moreover,
\[
\HH^0(C,T_S|_C)\cong
\ker\bigl(\cup\,s_C\bigr).
\]
In particular $\HH^1(S,T_S(-H))=0$ forces $\HH^0(C,T_S|_C)=0$ for every $C\in|H|_{\sm}$.
\end{lemma}

\begin{proof}
The differential $d\mu_{[C]}$ is the connecting map of the normal bundle sequence
\[
\begin{tikzcd}[column sep=small]
0 \arrow[r] & T_C \arrow[r] & T_S|_C \arrow[r] & N_{C/S} \arrow[r] & 0;
\end{tikzcd}
\]
as $\deg T_C=2-2g<0$ we have $\HH^0(C,T_C)=0$, so $\HH^0(C,T_S|_C)=\ker d\mu_{[C]}$ \cite[Lemma~3.1]{DH}. Tensoring $0\to\Oc_S(-H)\to\Oc_S\to\Oc_C\to0$ with $T_S$ and taking cohomology gives
\[
\begin{tikzcd}[column sep=small]
0=\HH^0(S,T_S) \arrow[r]
& \HH^0(C,T_S|_C) \arrow[r, "\partial"]
& \HH^1(S,T_S(-H)) \arrow[r, "\cup s_C"]
& \HH^1(S,T_S),
\end{tikzcd}
\]
where $\HH^0(S,T_S)=0$ on a K3. Thus $\partial$ identifies
\[
   \HH^0(C,T_S|_C) = \ker(\cup s_C).\qedhere
\]
\end{proof}

Let $S$ be a K3 surface, and let $H \in \Pic(S)$ be a globally generated ample line bundle. The vanishing of $\HH^0(C,T_S|_C)$ for a general curve $C \in |H|$ is proved in \cite{BA,DH}. If $S$ is primitively polarized or the degree of $H$ is sufficiently large, this vanishing holds for every $C \in |H|$ as a consequence of the vanishing
\[
H^1(S,\Omega_S(H))\simeq H^1(S,T_S(-H))=0.
\]
as shown in Lemma \ref{lem:kernel} above. Vanishing statements for cohomology groups of the form $H^1(S,\Omega_S(H))$ are commonly referred to as Bott vanishing.

For Picard-rank-one primitively polarised $K3$ surfaces of sufficiently high degree, the unramifiedness and quasi-finiteness of $\mu$ follow immediately from results of Totaro \cite{Totaro}.

\begin{proposition}\label{prop:good-range}
Let $S$ be a K3 surface with $\Pic(S)=\ZZ[H]$, where $H$ is the ample generator. If $H^2=20$ or $H^2\ge24$, then, for every $n\geq1$ and every smooth $C\in|nH|$, one has $\HH^0(C,T_S|_C)=0$. In particular, the moduli map $\mu_n\colon|nH|_{\sm}\to \Mmod_{g_n}$ is unramified.
\end{proposition}

\begin{proof}
By Totaro's Bott vanishing for $\rho=1$ K3 surfaces \cite[Theorems~3.2--3.3]{Totaro}, $\HH^1(S,\Om_S(nH))=0$ for every $n\geq1$, when $H^2=20$ or $H^2\ge24$ (and it fails for $H^2=22$). Serre duality gives $\HH^1(S,\Om_S(nH))\cong \HH^1(S,T_S(-nH))^{\vee}$, so $\HH^1(S,T_S(-nH))=0$ for every $n\geq1$, and \Cref{lem:kernel} applies with $nH$ in place of $H$.
\end{proof}

\begin{remark}\label{rem:badrange}
Still assuming $\Pic(S)=\ZZ[H]$, Bott vanishing fails when
$H^2=d\le18$ or $d=22$ \cite[ Theorem 3.1 - 3.2]{Totaro}. Thus for
$d\in\{2,4,\dots,18,22\}$ the argument of
\Cref{prop:good-range} alone does not determine whether, or for which
curves $C\in|H|$, the group $\HH^0(C,T_S|_C)$ is nonzero.
\end{remark}

The same is true for any K3, if the polarisation is sufficiently positive, in the following precise sense.

\begin{proposition}\label{prop:good-range2}
Let $S$ be any complex K3 surface, and let $H$ be an ample line bundle on $S$. Suppose that $H^2 \geq 74$. If there is no curve $E \subset S$ with $E^2=0$ and $1 \leq H \cdot E \leq 4$, then $\HH^0(C,T_S|_C)=0$ for every smooth curve $C\in|H|$. In particular, the moduli map $\mu\colon|H|_{\sm}\to \Mmod_g$ is unramified for every sufficiently positive line bundle $H$ on $S$.
\end{proposition}

\begin{proof}
The proof is identical to the previous one, except that it uses \cite[Theorem~5.1]{Totaro}. 
\end{proof}

Another proof of the vanishing $H^0(T_S|_C)=0$ follows from the stability of the restriction $T_S|_C$. We spell out this in the case of primitively polarized K3 surfaces.

\begin{proposition}[{\cite[Theorem 2.8]{Hein}, \cite{GO}, \cite{DH}}]\label{prop:stability}
Let $S$ be a K3 surface with $\Pic(S)=\ZZ[H]$. If $C\in |nH|$ is a smooth curve such that
\[
\frac{H^2}{2}(2n-1)>48,
\]
then the restriction $T_S|_C$ is $\mu$-stable. Hence $\HH^0(C,T_S|_C)=0$ and the moduli map is unramified at $[C]$.
\end{proposition}

As a consequence of the previous discussion, the interesting remaining  cases are multiples of low degree polarizations. When $H$ is ample of degree $2$, or very ample of degree $H^2 \geq 4$, Bott vanishing is completely determined by the following results of  \cite{DedieuSernesi} and \cite{KnutsenNormal}.

\begin{theorem} [{\cite[Lemma~3.5 and Theorem~3.6]{DedieuSernesi}, \cite[Lemma~2.3, (5), and Proposition~1.4]{KnutsenNormal}}]\label{prop:knutsen-normal}

Let $S$ be a K3 surface and $L$ be a very ample polarisation of
degree $L^2 \geq 4$, giving an embedding $S\subset\PP^g$. Then
\[
\HH^1(S,T_S(-nL))
\cong
\HH^0\bigl(S,N_{S/\PP^g}(-n)\bigr)
\qquad\text{for every }n\ge2.
\]
Moreover, $\HH^1(S,T_S(-nL))=0$ for all $n \geq 2$ except in the following cases.

\[
\begin{array}{@{}cccc@{}}
\toprule
L^2 & n
& h^1\!\left(S,T_S\otimes L^{-n}\right)
& \text{Notes} \\
\midrule
 4 & 2 & 10 & \text{any }(S,L) \\
 4 & 3 &  4 & \text{any }(S,L) \\
 4 & 4 &  1 & \text{any }(S,L) \\
 6 & 2 &  6 & \text{any }(S,L) \\
 6 & 3 &  1 & \text{any }(S,L) \\
 8 & 2 &  3 & \text{any }(S,L) \\
10 & 2 &  1 & \text{any }(S,L) \\
12 & 2 &  1 & (1) \\
14 & 2 &  1 & (2) \\
16 & 2 &  1 & (3) \\
18 & 2 &  1 & (4) \\
\bottomrule
\end{array}
\]
Here $(1),(2),(3),(4)$ have precise geometric descriptions (see \cite[Theorem~3.6]{DedieuSernesi}). The primitive cases among them occur on special Noether--Lefschetz loci. In particular, none of $(1),(2),(3),(4)$ can occur when $\Pic(S)=\ZZ[L]$. The 2-Veronese quartic among the cases in $(3)$ has $L=2D$, and the case in $(4)$ has $L=3D$, so these two cases are non-primitive.
\end{theorem}

\begin{lemma}[{\cite[Lemma~3.7]{DedieuSernesi}}]
Let $S$ be a $K3$ surface, and let $L$ be an ample and globally
generated line bundle satisfying $L^2=2$. Then the dimensions of
\[
H^1\!\left(S,T_S\otimes L^{-n}\right), \qquad n\geq 1,
\]
are given by the following table.
\[
\renewcommand{\arraystretch}{1.2}
\begin{array}{@{}cccc@{}}
\toprule
n & g(L^n) & \operatorname{Cliff}(L^n)
& h^1\!\left(S,T_S\otimes L^{-n}\right) \\
\midrule
1      &  2      & 0  & 18 \\
2      &  5      & 0  & 15 \\
3      & 10      & 2  & 10 \\
4      & 17      & >2 &  6 \\
5      & 26      & >2 &  3 \\
6      & 37      & >2 &  1 \\
n\geq7 & n^2+1   & >2 &  0 \\
\bottomrule
\end{array}
\]
\end{lemma}

In the primitive polarization situation, as an immediate consequence of the previous results we have the following.
\begin{corollary}\label{cor:knutsen-double}
Let $(S,H)$ be a K3 surface with $\Pic(S)=\ZZ[H]$. Then
\[
\HH^0(C,T_S|_C)=0
\]
for every smooth curve $C\in|nH|$ in each of the following ranges.
\[
\begin{array}{c|cccccc}
H^2 & 2 & 4 & 6 & 8,\ldots,18 & 22 & \geq 20,\ \neq 22\\
\hline
n   & \geq 7 & \geq 5 & \geq 4 & \geq 3 & \geq 3 & \geq 1
\end{array}
\]
Consequently, the moduli map associated with $|nH|$ is unramified everywhere on its smooth locus in these ranges.
\end{corollary}

\begin{remark} 
\label{remark-quasifiniteness}
When $S$ is a Picard-rank-one K3 surface, \Cref{cor:knutsen-double} implies that the map is quasi-finite in these ranges. Quasi-finiteness for curves in the primitive polarisation follows instead from \cite[Theorem~A]{CG}. More generally, we have quasi-finiteness for every polarised K3 surface $(S,H)$ whenever $H^2 \geq 20$ and $\operatorname{Cliff}(H) \geq 2$ \cite[Proposition~8.6]{CDS}.
\end{remark}

For a primitively polarised K3 surface $(S,H)$ there are also strong injectivity results. More precisely, if $H^2 \geq 4$, then the
Mukai program, initiated in \cite{mukai} and carried out by Feyzbakhsh in \cite{FeyzbakhshI,FeyzbakhshII}, implies for $g=11$ and $g\geq13$ that the moduli map $|H|_{\sm} \rightarrow \mathcal{M}_g$ is injective. More generally, Cheng--Li--Wu prove the corresponding reconstruction for every irreducible $C\in|nH|$ provided $ng\geq11$ and $ng\neq12$ \cite[Theorem~1.1]{CLW}. It follows that $\mu_n$ is injective in this range.

\begin{remark}
The exclusion of degree $2$ is essential. If $f\colon S\to\PP^2$ is the double cover defined by $|H|$ and $\iota$ is its covering involution, then a general smooth curve $C\in|nH|$ is distinct from, but isomorphic to, $\iota(C)$; hence $\mu_n$ is not injective, even when $\Pic(S)=\ZZ[H]$.

More generally, fibres of $\mu_n$ contain the orbits of the polarised automorphism group, so injectivity can fail at higher Picard rank. For example, on the Fermat quartic
\[
   S_F=\{x_0^4+x_1^4+x_2^4+x_3^4=0\}\subset\PP^3,
\]
the automorphism exchanging $x_0$ and $x_1$ acts nontrivially on every $|\Oc_{S_F}(n)|$. A general smooth member and its translate are therefore distinct isomorphic curves, so none of the corresponding maps $\mu_n$ is injective.
\Cref{ex:quarticexamples} moreover exhibits an isotrivial pencil in
$|H|$, so $\mu_1$ is not quasi-finite.
\end{remark}

\begin{table}[htbp]
\centering
\footnotesize
\renewcommand{\arraystretch}{1.2}
\begin{tabularx}{\textwidth}{@{}>{\raggedright\arraybackslash}p{0.18\textwidth}
                                >{\raggedright\arraybackslash}X
                                >{\raggedright\arraybackslash}X
                                >{\raggedright\arraybackslash}X@{}}
\toprule
$H^2$ & Quasi-finite & Unramified & Injective \\
\midrule
$2$
& $n=1$ or $n\ge4$.
& Yes for $n\ge7$; fails for $n=1$ for general $S$, ramified at exactly $171$ points and for $n=6$, ramified along a pos.-diml. locus.
& Fails for $n\ge3$. \\
\addlinespace
$4$
& $n=1$ or $n\ge3$.
& Yes for $n=1$ or $\ge5$ for general $S$; fails for $n=4$, ramified along a pos.-diml. locus.
& $n\ge5$. \\
\addlinespace
$6$
& $n\ge1$.
& $n\ge4$.
& $n\ge4$. \\
\addlinespace
$8,10$
& $n\ge1$.
& $n\ge3$.
& $n\ge3$. \\
\addlinespace
$12,\ldots,18,22$
& $n\ge1$.
& $n\ge2$.
& $n\ge2$. \\
\addlinespace
$20$ or $\ge24$
& $n\ge1$.
& $n\ge1$.
& $n\ge1$. \\
\bottomrule
\end{tabularx}
\vspace{0.5em}
\caption{Known behaviour of $\mu_n$ for K3 surfaces with $\Pic(S)=\ZZ[H]$, where $H$ is the primitive ample generator. Unless explicitly qualified otherwise, the positive quasi-finiteness, unramifiedness and injectivity statements hold for every such surface and every smooth member of the indicated linear system.
Quasi-finiteness follows from \cite{CG,CDS} (see Remark \ref{remark-quasifiniteness}); the unramifiedness statements combine
\cite{Totaro,DedieuSernesi,KnutsenNormal} with
\Cref{thm:very-general_h2}, \Cref{cor:quartic-unramified},
\Cref{cor:knutsen-double}, \Cref{sec:stability-degree2} and
\Cref{sec:stability-degree4};
injectivity follows from \cite{FeyzbakhshI,FeyzbakhshII,CLW}.}
\label{tab:moduli-map-summary}
\end{table}

To address what remains open from the three guiding questions in the introduction:

\begin{description}[style=sameline,leftmargin=2em,labelsep=.5em,
    itemsep=.4\baselineskip,topsep=.5\baselineskip]
\item[\textbf{(Finiteness)}] Within the Picard-rank-one setting, the quasi-finiteness question remains open
only for (see \cite[\S 3.1]{CG})
\[
   (H^2,n)=(2,2),\ (2,3),\ (4,2).
\]
\item[\textbf{(Injectivity)}] Injectivity in Picard rank one remains open for
$n=1,2$ when $H^2=2$; $1\le n\le4$ when $H^2=4$;
$1\le n\le3$ when $H^2=6$; $n=1,2$ when $H^2=8$ or $10$;
and $n=1$ when $H^2\in\{12,14,16,18,22\}$. 
\item[\textbf{(Unramifiedness)}] The assertions for $n=1$ in degrees $2$ and $4$ in \Cref{tab:moduli-map-summary} are for a general surface. Subject to this
the open cases for everywhere unramifiedness are $2\le n\le5$ when
$H^2=2$; $n=2,3$ when $H^2=4$; $1\le n\le3$ when $H^2=6$;
$n=1,2$ when $H^2=8$ or $10$; and $n=1$ when
$H^2\in\{12,14,16,18,22\}$.
\end{description}

\section{General hypersurfaces}\label{sec:hypersurfaces}

The study of moduli-theoretic variation of hyperplane sections  has a long history. See, for example, \cite{HMP,VanOpstallVeliche,PatelRiedlTseng,bricalli2026maximalvariationproblemlefschetz}. Recently, Beauville in \cite{BE}, expanded on the connection between maximal variation and Lefschetz properties of Jacobian rings (see also  \cite{IlardiNasrollahTafazolian} for related results). See also \cite[Section~6.3]{Bertsch} for related explicit algebraic calculations regarding quartic surfaces. We mostly follow Beauville \cite{BE}, though our aims are slightly different. In most of the aforementioned papers, the hypersurface $X$ is arbitrary and one derives implications about the \textit{general} hyperplane section $Y$ of $X$, whereas here we start with a general $X$ and derive implications for \textit{every} smooth $Y$. For $r\ge1$ and $k\ge0$, set
\[
   V_r:=\CC[x_0,\dots,x_{r-1}],
   \qquad V_{r,k}:=(V_r)_k.
\]
We now consider the ramification of the moduli map for smooth hypersurfaces of dimension $n\geq3$ and degree $d\geq4$. Fix coordinates $x_0, \dots, x_n$ on $\PP^n$ and let $f \in V_{n+1,d}$ be the equation of a smooth hypersurface $X$ of degree $d$. For a smooth hyperplane section $Y = X \cap \{L=0\}$, we may change coordinates so that $L=x_n$. Let
\begin{align*}
   g&:=f|_{x_n=0}\in V_{n,d} && (\text{equation of } Y \text{ in } \{L=0\}),\\
   h&:=\partial_n f|_{x_n=0}\in V_{n,d-1} && (\text{transverse derivative}).
\end{align*}
There are two Jacobian rings that govern the deformation theory, both of which are graded Artinian Gorenstein rings equipped with perfect pairings:
\begin{enumerate}
    \item That of the smooth hyperplane section $Y$,
    \[
        R_{g} = \CC[x_0, \dots, x_{n-1}] / J_{g}, \quad J_{g} = (\partial_0 g, \dots, \partial_{n-1} g).
    \]
    \item That of the hypersurface $X$,
    \[ 
        R_f = \CC[x_0, \dots, x_n] / J_f, \quad J_f = (\partial_0 f, \dots, \partial_n f).
    \]
\end{enumerate}

For a smooth degree-$d$ hypersurface in $\PP^{r-1}$, the partial derivatives form a regular sequence of $r$ forms of degree $d-1$. Hence its Jacobian ring $R$ has Hilbert series
\[
\operatorname{Hilb}_R(t)
:=\sum_{m\ge0}\dim(R_m)t^m
=\frac{(1-t^{d-1})^r}{(1-t)^r}
=(1+t+\cdots+t^{d-2})^r.
\]
Its socle degree is the largest $s$ for which $R_s\neq0$, namely $s=r(d-2)$; the top piece $R_s$ is one-dimensional and multiplication gives perfect pairings $R_m\otimes R_{s-m}\to R_s\cong\CC$.

For instance, when $X=S\subset\PP^3$ is a quartic surface, $R_g$ is the Jacobian ring of a smooth plane quartic and $R_f$ is a complete intersection of four cubics. Their Hilbert series are
\begin{align*}
\operatorname{Hilb}_{R_g}(t)
&=(1+t+t^2)^3
=1+3t+6t^2+7t^3+6t^4+3t^5+t^6,\\
\operatorname{Hilb}_{R_f}(t)
&=(1+t+t^2)^4
=1+4t+10t^2+16t^3+19t^4+16t^5+10t^6+4t^7+t^8.
\end{align*}
Thus their socle degrees are $6$ and $8$, respectively.

The proof of the following is taken more or less verbatim from Beauville's paper.

\begin{proposition}[{\cite[Proposition~2]{BE}}]\label{prop:beauville}
Let $X=\{f=0\}\subset\PP^n$ ($n\ge3$) be a smooth hypersurface of degree $d\ge3$, with $d\ge4$ if $n=3$, and let $Y=X\cap\{L=0\}$ be a smooth hyperplane section. Multiplication by $L$ gives a map
\[
\begin{tikzcd}
R_{f,d-1} \arrow[r, "\cdot L"] & R_{f,d}
\end{tikzcd}
\]
and we have a canonical isomorphism
\[
\HH^0(Y,T_X|_Y)
\cong\ker\bigl(\cdot L\colon R_{f,d-1}\to R_{f,d}\bigr),
\]
where $R_f$ is the Jacobian ring of $X$.
\end{proposition}

\begin{proof}
Choosing coordinates so that $L=x_n$, $Y$ is defined by $x_n=f=0$. From the normal-bundle sequence
\[
\begin{tikzcd}[column sep=small]
0 \arrow[r] & T_X|_Y \arrow[r] & T_{\PP^n}|_Y
\arrow[r, "df"] & \Oc_Y(d) \arrow[r] & 0,
\end{tikzcd}
\]
the global sections of $T_X|_Y$ identify with the kernel of $df$.

Using the Euler sequence on $\PP^n$ restricted to $Y$, the restriction map $\HH^0(\PP^n,T_{\PP^n})\to\HH^0(Y,T_{\PP^n}|_Y)$ is surjective, because $\HH^1(Y,\Oc_Y)=0$ (for $n\ge4$) or its map to $\HH^1(Y,\Oc_Y(1)^{\oplus(n+1)})$ is injective (for $n=3$, $d\ge4$). Thus any section of $T_{\PP^n}|_Y$ lifts to a linear vector field $D=\sum_{i=0}^n L_i\partial_i$ with $L_i\in V_{n+1,1}$, uniquely defined modulo the Euler vector field $\sum_{i=0}^n x_i\partial_i$.

The condition $df(D)=0$ on $Y$ means that $\sum_{i=0}^n L_i\partial_i f$ lies in the degree-$d$ part of the ideal $(x_n,f)$, so there exist $E\in V_{n+1,d-1}$ and $a\in\CC$ such that
\[
   \sum_{i=0}^n L_i\partial_i f = x_n E + af.
\]
By Euler's identity $d\cdot f=\sum_{i=0}^n x_i\partial_i f$, the term $af$ lies in the Jacobian ideal $J_f$. Projecting the equation to the degree-$d$ piece of the Jacobian ring $R_f$, the left-hand side vanishes, giving $x_n E=0$ in $R_{f,d}$. This determines a class
\[
[E]\in\ker\bigl(\cdot x_n\colon R_{f,d-1}\to R_{f,d}\bigr),
\]
which is nonzero when the original section is nonzero.

Conversely, given $E\in V_{n+1,d-1}$ such that $x_n E\in (J_f)_d$, we may write $x_n E=\sum_{i=0}^n L_i\partial_i f$ for some linear forms $L_i$. Adjusting $E$ modulo $(J_f)_{d-1}$, we can assume the $L_i$ are independent of $x_n$. Then $D=\sum_{i=0}^n L_i\partial_i$ defines a vector field mapping to $0$ under $df$. One verifies this gives a well-defined isomorphism.
\end{proof}

For the genericity theorem below, we will need the following alternative description of the above dimension. This is very similar in spirit to \cite[Proposition~3.1]{IlardiNasrollahTafazolian}.

\begin{lemma}\label{lem:criterion}
Let $X=\{f=0\}\subset\PP^n$ ($n\ge3$) be a smooth hypersurface of degree $d\ge3$, with $d\ge4$ if $n=3$, and $Y=X\cap\{L=0\}$ a smooth hyperplane section. Choosing coordinates so that $L=x_n$, let $g=f|_{x_n=0}$ be the equation of $Y$ in $\PP^{n-1}$ and $h=\partial_n f|_{x_n=0}$ the transverse derivative. Multiplication by $\bar h$ gives
\[
\begin{tikzcd}
R_{g,1} \arrow[r, "\cdot\bar h"] & R_{g,d}.
\end{tikzcd}
\]
Then
\[
\dim\HH^0(Y,T_X|_Y)
=n-\rank\bigl(\cdot\bar h\colon R_{g,1}\to R_{g,d}\bigr),
\]
and hence $\HH^0(Y,T_X|_Y)\neq0$ if and only if there exists a nonzero linear form $\lambda\in R_{g,1}$ such that $\lambda h\in J_{g}$. 
\end{lemma}

\begin{proof}
By \Cref{prop:beauville}, multiplication by $x_n$ gives
\[
\begin{tikzcd}
R_{f,d-1} \arrow[r, "\cdot x_n"] & R_{f,d}.
\end{tikzcd}
\]
Moreover,
\[
\HH^0(Y,T_X|_Y)
\cong\ker\bigl(\cdot x_n\colon R_{f,d-1}\to R_{f,d}\bigr).
\]
Modulo $x_n$, the $n+1$ generators of the Jacobian ideal $J_f$ become $h$ and $\partial_0 g,\dots,\partial_{n-1} g$, so 
\[
   R_f/x_n R_f\cong R_{g}/(\bar h).
\] 
Taking the degree-$d$ part of the exact sequence
\[
\begin{tikzcd}[column sep=small]
0 \arrow[r] & K \arrow[r] & R_f(-1) \arrow[r, "\cdot x_n"]
& R_f \arrow[r] & R_g/(\bar h) \arrow[r] & 0
\end{tikzcd}
\]
for $K:=\ker(\cdot x_n)$, gives $\dim \HH^0(Y,T_X|_Y) = \dim K_d = \dim R_{f,d-1} - \dim R_{f,d} + \dim(R_{g}/(\bar h))_d$. Write
\[
   \rho_n(k):=\dim_{\CC}R_{f,k},
   \qquad
   \rho_{n-1}(k):=\dim_{\CC}R_{g,k}
\]
for their Hilbert functions. Using
\[
\dim(R_g/(\bar h))_d
=\rho_{n-1}(d)-\rank(\cdot\bar h),
\]
we obtain
\[
\dim \HH^0(Y,T_X|_Y)
=\rho_n(d-1)-\rho_n(d)+\rho_{n-1}(d)-\rank(\cdot\bar h).
\]
For $m\in\{n-1,n\}$, write
\[
   H_m(t):=\sum_{k\ge0}\rho_m(k)t^k=(1+t+\dots+t^{d-2})^{m+1}.
\]
The identity $H_m(t) = (1+t+\dots+t^{d-2})H_{m-1}(t)$ implies $\rho_n(d) - \rho_n(d-1) = \rho_{n-1}(d) - \rho_{n-1}(1)$. Since $d\ge 3$, the generators of the Jacobian ideal of $g$ have degree $d-1\ge 2$, so in degree $1$ there are no relations and $\rho_{n-1}(1) = \dim V_{n,1} = n$. Therefore
\[
   \dim \HH^0(Y,T_X|_Y)=n-\rank(\cdot\bar h).
\]

Finally, $\HH^0(Y,T_X|_Y)\neq0$ if and only if $\rank(\cdot\bar h)<n$, i.e.,\ some nonzero $\lambda\in R_{g,1}$ has $\lambda\bar h=0$ in $R_{g,d}$, that is $\lambda h\in J_{g}$.
\end{proof}

In light of \Cref{lem:criterion}, we wish to bound the locus of pairs $(g,h)$ that admit ramification, where $g \in V_{n,d}$ plays the role of the hyperplane section equation and $h \in V_{n,d-1}$ the transverse derivative of the underlying hypersurface $f$. We start with the following technical lemma.

\begin{lemma}\label{lem:incidence-bound}
Let \(n\ge3\) and \(d\ge4\). For
\[
D=\sum_{i=0}^{n-1}A_i\partial_i,
\qquad A_i\in V_{n,1},
\]
consider the linear map
\[
\begin{tikzcd}[row sep=small]
V_{n,d}\oplus V_{n,d-1}
  \arrow[r,"M_D"]
& V_{n,d},\\
(g,h)\arrow[r,mapsto]
&D(g)-x_{n-1}h,
\end{tikzcd}
\]
and the  incidence variety
\[
\mathcal I_{n,d}:=
\left\{
(g,h,D)\in
V_{n,d}\times V_{n,d-1}\times(V_{n,1})^n
\ \middle|\ M_D(g,h)=0
\right\}.
\]
Then
\[
\dim\mathcal I_{n,d}\le \dim V_{n,d}+n^2-n.
\tag{1}\label{eq:I-general-bound}
\]
In the exceptional case \((n,d)=(3,4)\), one has the stronger estimate
\[
\dim\mathcal I_{3,4}\le19.
\tag{2}\label{eq:I-exceptional-bound}
\]
\end{lemma}

\begin{proof}
We will estimate the dimension of \(\mathcal I_{n,d}\) using the projection
\[
\begin{aligned}
\pi\colon\mathcal I_{n,d}&\longrightarrow(V_{n,1})^n,\\
(g,h,D=\textstyle\sum A_i\partial_i)&\longmapsto(A_0,\dots,A_{n-1}).
\end{aligned}
\]
The fibre of \(\pi\) over \(D\) is exactly \(\ker M_D\). We first describe in a convenient way the kernel of $M_D$. The restriction of \(M_D\) to \(0\oplus V_{n,d-1}\) maps isomorphically
onto
\[
x_{n-1}V_{n,d-1}\subseteq V_{n,d}.
\]
Thus this part contributes \(\dim V_{n,d-1}\) to the rank. The remaining
contribution is the dimension of the image in $V_{n,d}/x_{n-1}V_{n,d-1} \simeq V_{n-1,d}$, that is we need to compute the dimension of the image of the linear map
\[
\begin{tikzcd}[row sep=small]
\phi_D\colon V_{n,d}\oplus V_{n,d-1}
  \arrow[r]
&V_{n,d}/x_{n-1}V_{n,d-1}\simeq V_{n-1,d},\\
(g,h)\arrow[r,mapsto]
&D(g)|_H,
\end{tikzcd}
\]
where $H:=\{x_{n-1}=0\}$.
 Set
\[
\delta_D:=
\sum_{i=0}^{n-2}(A_i|_H)\partial_i
\in \operatorname{Der_{\CC}{}}(V_{n-1})_0,
\qquad
a_D:=A_{n-1}|_H\in V_{n-1,1}.
\]
Here $\operatorname{Der}_{\CC}(V_{n-1})_0$ denotes the space of $\CC$-linear derivations of $V_{n-1}$ that preserve degree. Restriction to $V_{n-1,1}$ gives an isomorphism $\operatorname{Der}_{\CC}(V_{n-1})_0\simeq\operatorname{End}(V_{n-1,1})$, whose inverse is
\begin{equation}
\label{iso derivation}
\operatorname{End}_{\CC}(V_{n-1,1})
\xrightarrow{\sim}
\operatorname{Der}_{\CC}(V_{n-1})_0,
\qquad
T\longmapsto
D_T:=\sum_{i=0}^{n-2}T(x_i)\partial_i.
\end{equation}
In the following we switch freely between these two descriptions. Writing
\[
g=b_0+x_{n-1}b_1+x_{n-1}^2b_2+\cdots,
\qquad
b_k\in V_{n-1,d-k},
\]
we obtain
\[
D(g)|_H=\delta_D(b_0)+a_Db_1.
\]
Consider the linear map
\[
\Phi_{\delta_D,a_D}\colon V_{n-1,d} \oplus V_{n-1,d-1}  \rightarrow V_{n-1,d}, \quad  \quad (b_0,b_1) \rightarrow \delta_D(b_0) 
 +a_D(b_1)
\]
and set $r(\delta_D,a_D):=\rank\Phi_{\delta_D,a_D}$. To be more precise here, we are thinking $\delta_D \in \operatorname{Der}_{\CC}(V_{n-1})_0$ (via \eqref{iso derivation}), and restricting its action on $V_{n-1,d}$. Since the map
\[
V_{n,d}\longrightarrow V_{n-1,d}\oplus V_{n-1,d-1},
\qquad
g\longmapsto(b_0,b_1),
\]
is surjective, the rank of $\phi_D$ is \(r(\delta_D,a_D)\). Then
\[
\rank M_D=\dim V_{n,d-1}+r(\delta_D,a_D)
\tag{3}\label{eq:rank-MD-two-parts}
\]
and
\[
\begin{aligned}
\dim\ker M_D
&=\dim V_{n,d}+\dim V_{n,d-1}-\rank M_D\\
&=\dim V_{n,d}-r(\delta_D,a_D).
\end{aligned}
\tag{4}\label{eq:kernel-MD-two-parts}
\]
We write now the parameter space $(V_{n,1})^n$ as a disjoint union of locally closed subvarieties for which we are able to   estimate \(r(\delta_D,a_D)\). For this purpose, consider the linear map
\begin{equation}
\label{parametersmap}
\begin{aligned}
(V_{n,1})^n
&\longrightarrow \operatorname{End}(V_{n-1,1})\oplus V_{n-1,1}\\
D&\longmapsto(\delta_D,a_D).
\end{aligned}
\end{equation}
It is surjective and has \(n\)-dimensional kernel. We decompose the parameter space $(V_{n,1})^n$ into the inverse images of the three loci in $\operatorname{End}(V_{n-1,1})\oplus V_{n-1,1}$ defined by $a\ne0$, by $a=0$ and $\delta\ne0$, and by $a=\delta=0$. If $a\ne0$, multiplication by $a$ identifies $V_{n-1,d-1}$ with the subspace $aV_{n-1,d-1}\subset V_{n-1,d}$, so $r(\delta,a)\ge \dim V_{n-1,d-1}$; using also \eqref{parametersmap} we find that  the inverse image of this locus in $(V_{n,1})^n$ has dimension $n(n-1)+n=n^2$. If $a=0$ and $\delta\ne0$, then $r(\delta,0)\ge1$: indeed, choose $\ell\in V_{n-1,1}$ with $\delta(\ell)\ne0$, and then $\delta(\ell^d)=d\ell^{d-1}\delta(\ell)\ne0$. The inverse image of this  locus has dimension $n+(n-1)^2=n^2-n+1$. Finally, if $a=\delta=0$, then $r(\delta,a)=0$ and the corresponding locus has dimension $n$. 

Consequently, the dimensions of the parts of $\mathcal I_{n,d}$ lying  over these three strata are bounded from above respectively by
\[
   \dim V_{n,d}+n^2-\dim V_{n-1,d-1},\qquad
   \dim V_{n,d}+n^2-n,\qquad
   \dim V_{n,d}+n.
\]
Since $\dim V_{n-1,d-1}\ge n$, all three dimensions are at most $\dim V_{n,d}+n^2-n$. Hence
\[
\dim\mathcal I_{n,d}\le\dim V_{n,d}+n^2-n.
\]
This proves \eqref{eq:I-general-bound}. Now we consider the case $(n,d)=(3,4)$. In this situation
\[
\dim V_{3,4}=15,
\qquad
\dim V_{2,3}=4,
\qquad
\dim V_{2,4}=5.
\]
We want to give a better bound for $\operatorname{dim}(\mathcal{I}_{3,4})$. For $(\delta,a) \in \operatorname{End}(V_{2,1}) \oplus V_{2,1}$ we consider as before
\[
\Phi_{\delta,a}\colon V_{2,4} \oplus V_{2,3} \rightarrow V_{2,4}.
\]

As above, if $a\ne0$, then $aV_{2,3}\subset\operatorname{im}\Phi_{\delta,a}$ has dimension $4$. Since $\dim V_{2,4}=5$, the rank of $\Phi_{\delta,a}$ is therefore either $4$ or $5$. Thus we further stratify the locus $\{a \neq 0\}\subset \operatorname{End}(V_{2,1}) \oplus V_{2,1} $ as $T_4 \sqcup T_5$,
where for $r\in\{4,5\}$,  we set
\[
T_r:=
\left\{
(\delta,a)\in
\operatorname{End}(V_{2,1})\times
\bigl(V_{2,1}\setminus\{0\}\bigr)
\ \middle|\
\rank\Phi_{\delta,a}=r
\right\}.
\]
We denote by $S_4$ and $S_5$ the inverse images via \eqref{parametersmap}, namely
\[
   S_r:=\{D\in(V_{3,1})^3\mid a_D\ne0,\
   \rank\Phi_{\delta_D,a_D}=r\}.
\]
Since $\dim S_5\leq \operatorname{dim}(V_{3,1})^3=9$, we have
\[
   \dim\{(g,h,D) \in \mathcal{I}_{3,4}\mid D\in S_5\}
   \le\dim V_{3,4}+9-5=\dim V_{3,4}+4=19.
\]
Now we need to bound the dimension of $S_4$. On $S_4$ the composite
\[
\begin{tikzcd}
V_{2,4} \arrow[r, "\delta"] & V_{2,4} \arrow[r] & V_{2,4}/aV_{2,3}
\end{tikzcd}
\]
vanishes. For fixed $a\ne0$ this is a nontrivial linear condition on $\delta$. In fact, choosing $\delta=x_0\partial_{x_0}+x_1\partial_{x_1}$,  one sees that it acts on $V_{2,4}$ as multiplication by $4$ and then does not
map all of $V_{2,4}$ into the proper subspace $aV_{2,3}$. In other words $T_4$ is a proper closed subset of the irreducible space $\operatorname{End}(V_{2,1})\times
\bigl(V_{2,1}\setminus\{0\})$. Hence its dimension is $\leq 5$. Thus $\operatorname{dim}(S_4) \leq 8$, and 
\begin{align*}
   \dim\{(g,h,D) \in \mathcal{I}_{3,4} \mid D\in S_4\}
   &\le\dim V_{3,4}+8-4\\
   &=\dim V_{3,4}+4=19.
\end{align*}
Now suppose \(a=0\) and \(\delta\neq0\). We claim that
\[
\rank\bigl(\delta\colon V_{2,4}\longrightarrow V_{2,4}\bigr)\ge4.
\]
As usual, recall that  $\delta  \in \operatorname{End}(V_{2,1}) \simeq \operatorname{Der}_{\CC}(V_2)_0$ (via the isomorphism \eqref{iso derivation}). Here we are considering the restriction of its action on $V_{2,4}$. To understand its  rank we relate the eigenvalues on  $V_{2,1}$ to the ones in $V_{2,4}$.  By the Jordan normal form, we may choose a basis \(u,v\) of
\(V_{2,1}\) such that
\[
\delta(u)=\alpha u,
\qquad
\delta(v)=\gamma u+\beta v.
\]
Here  \(\alpha\) and \(\beta\) are its eigenvalues, counted with
multiplicity. Using \eqref{iso derivation} one sees that  the associated derivation is determined by the Leibniz
rule
\[
D_\delta(\ell_1\ell_2\ell_3\ell_4)
=
\sum_{j=1}^4
\ell_1\cdots\ell_{j-1}\delta(\ell_j)
\ell_{j+1}\cdots\ell_4.
\]
Applying this formula to the monomial \(u^{4-i}v^i\), for
\(0\le i\le4\), gives
\[
\begin{aligned}
\delta|_{V_{2,4}}(u^{4-i}v^i)
&=(4-i)u^{3-i}v^i\delta(u)
  +i u^{4-i}v^{i-1}\delta(v)\\
&=(4-i)\alpha u^{4-i}v^i
  +i u^{4-i}v^{i-1}(\gamma u+\beta v)\\
&=\bigl((4-i)\alpha+i\beta\bigr)u^{4-i}v^i
  +i\gamma u^{5-i}v^{i-1}.
\end{aligned}
\]
Therefore, with respect to the ordered basis $\{u^4, u^3v, u^2v^2, uv^3,v^4\}$ of $V_{2,4}$, the matrix of $\delta|_{V_{2,4}}$ is upper triangular, and its eigenvalues are
\[
4\alpha,\qquad
3\alpha+\beta,\qquad
2\alpha+2\beta,\qquad
\alpha+3\beta,\qquad
4\beta.
\]
Thus
\[
   \det(\delta|_{V_{2,4}})
   =\prod_{i=0}^4\bigl((4-i)\alpha+i\beta\bigr).
\]
If $(\alpha,\beta)\ne(0,0)$, at most one factor in this product vanishes, so $\delta|_{V_{2,4}}$ has rank $4$ or $5$. It remains to consider the case $
\alpha=\beta=0.$ Since \(\delta\neq0\), the endomorphism \(\delta\) is nonzero
nilpotent. By the Jordan normal form, after rescaling \(v\), we may
assume
\[
\delta(u)=0,
\qquad
\delta(v)=u.
\]
The preceding formula then becomes
\[
\delta|_{V_{2,4}}(u^{4-i}v^i)
=
i\,u^{5-i}v^{i-1}.
\]
For \(i=0\), the image is zero, while the four images for $1\le i\le4$ are linearly independent. Hence
$\rank\delta|_{V_{2,4}}=4$.  We conclude that, for every nonzero
\(\delta\in\operatorname{End}(V_{2,1})\), $\rank\bigl(\delta|_{V_{2,4}})\ge4$ as desired. Therefore
\begin{align*}
   \dim\{(g,h,D) \in \mathcal{I}_{3,4}\mid a_D=0,\ \delta_D\ne0\}
   &\le\dim V_{3,4}+7-4\\
   &=\dim V_{3,4}+3=18.
\end{align*}
Finally, the locus $a=\delta=0$ has dimension $3$ and $r(0,0)=0$, so its corresponding incidence space also has dimension at most $\dim V_{3,4}+3=18$.
\end{proof}

\begin{proposition}\label{prop:strata-two-parts}
Let \(n\ge3\) and \(d\ge4\). For \(g\in V_{n,d}\), with Jacobian ideal
\[
J_g=(\partial_0g,\dots,\partial_{n-1}g),
\]
consider the incidence correspondence
\[
B_{n,d}:=
\left\{(g,h) \in V_{n,d}\times V_{n,d-1}\ \middle|\
\exists\lambda\in V_{n,1}\smallsetminus\{0\},\
\lambda h\in J_g
\right\}.
\]
Then every irreducible component \(Z\) of the closure
\(\overline{B_{n,d}}\subseteq V_{n,d}\times V_{n,d-1}\) satisfies
\[
\codim\bigl(Z,V_{n,d}\times V_{n,d-1}\bigr)>n.
\]
\end{proposition}

\begin{proof}
Introduce
\[
\widetilde B_{n,d}:=
\left\{
(g,h,[\lambda])
\in V_{n,d}\times V_{n,d-1}\times\PP(V_{n,1})
\ \middle|\ \lambda h\in J_g
\right\}.
\]
By Chevalley's theorem, $\widetilde B_{n,d}$ is constructible: on each standard affine chart of $\PP(V_{n,1})$, it is the projection of the closed locus obtained by adjoining linear forms $A_i$ and imposing $\lambda h=\sum_i A_i\partial_i g$. Thus it is a finite union of
locally closed subsets. Taking the irreducible components of each of
these locally closed subsets, we obtain a finite decomposition
\[
\widetilde B_{n,d}=\bigcup_{\alpha=1}^s S_\alpha
\]
into irreducible locally closed subsets. Let
\[
p\colon\widetilde B_{n,d}\longrightarrow\PP(V_{n,1})
\]
be the natural projection, and put
\[
Z_\alpha:=\overline{p(S_\alpha)}
\subseteq\PP(V_{n,1}).
\]
Then
\[
\dim S_\alpha
\le \dim Z_\alpha+
  \sup_{[\lambda]\in p(S_\alpha)}
  \dim\bigl( p^{-1}([\lambda])\bigr) \leq n-1 +  \sup_{[\lambda]\in p(S_\alpha)}
  \dim\bigl( p^{-1}([\lambda])\bigr) 
\]
By $\mathrm{PGL}_n$-transitivity, all the fibres
of \(p\) are isomorphic. Thus, for every \([\lambda]\),
\[
\dim p^{-1}([\lambda])
=\dim p^{-1}([x_{n-1}])
=:\dim(\widetilde B_{n,d})_{[x_{n-1}]}.
\]
Then every \(S_\alpha\) satisfies
\[
\dim S_\alpha
\le(n-1)+\dim(\widetilde B_{n,d})_{[x_{n-1}]}.
\tag{11}\label{eq:total-from-one-fibre-two-parts}
\]

We now bound the fibre over \([x_{n-1}]\). Let
\[
\mathcal I_{n,d}=
\left\{
(g,h,D)\in
V_{n,d}\times V_{n,d-1}\times(V_{n,1})^n
\ \middle|\ D(g)-x_{n-1}h=0
\right\}
\]
be the closed incidence variety of Lemma~\ref{lem:incidence-bound}.
Projection onto the first two factors maps \(\mathcal I_{n,d}\) onto
\[
(\widetilde B_{n,d})_{[x_{n-1}]}
=
\left\{
(g,h)\ \middle|\ x_{n-1}h\in(J_g)_d
\right\}.
\]
Consequently,
\[
\dim(\widetilde B_{n,d})_{[x_{n-1}]}
\le\dim\mathcal I_{n,d}.
\tag{12}\label{eq:projection-I-two-parts}
\]
Lemma~\ref{lem:incidence-bound} therefore gives
\[
\dim(\widetilde B_{n,d})_{[x_{n-1}]}
\le\dim V_{n,d}+n^2-n.
\tag{13}\label{eq:fixed-fibre-general}
\]
Combining \eqref{eq:total-from-one-fibre-two-parts} and
\eqref{eq:fixed-fibre-general}, we obtain, for every \(\alpha\),
\[
\dim S_\alpha\le\dim V_{n,d}+n^2-1.
\]
It follows  that every
irreducible component of \(\overline{\widetilde B_{n,d}}\) has dimension
at most
\[
\dim V_{n,d}+n^2-1.
\]
The projection to $V_{n,d}\times V_{n,d-1}$ is proper, so the image of $\overline{\widetilde B_{n,d}}$ is closed and contains $\overline{B_{n,d}}$. It follows that every irreducible component $Z$ of $\overline{B_{n,d}}$ satisfies
\[
\codim\bigl(Z,V_{n,d}\times V_{n,d-1}\bigr)
\ge\dim V_{n,d-1}-n^2+1.
\tag{17}\label{eq:component-bound-two-parts}
\]

If \(n\ge4\), then
\[
\dim V_{n,d-1}\ge\dim V_{n,3}
=\binom{n+2}{3}>n^2+n-1,
\]
where
\[
\binom{n+2}{3}-(n^2+n-1)
=\frac{n(n+1)(n-4)}6+1>0.
\]
Thus the right-hand side of \eqref{eq:component-bound-two-parts} is
greater than \(n\). If \(n=3\) and \(d\ge5\), then
\[
\dim V_{3,d-1}=\binom{d+1}{2}>11.
\]
Again, the right-hand side of \eqref{eq:component-bound-two-parts} is
greater than \(3\). It remains to consider \((n,d)=(3,4)\). In this case the stronger estimate
\eqref{eq:I-exceptional-bound} and \eqref{eq:projection-I-two-parts} gives
\[
\dim(\widetilde B_{3,4})_{[x_2]}\le19.
\]
The same argument as above then shows that every
irreducible component \(Z\) of \(\overline{B_{3,4}}\) has dimension at
most \(21\). Since
\[
\dim(V_{3,4}\times V_{3,3})=15+10=25,
\]
we conclude that
\[
\codim\bigl(Z,V_{3,4}\times V_{3,3}\bigr)
\ge25-21=4>3.
\]
\end{proof}
We now finally come to the proof of the first main theorem.
\begin{theorem}\label{thm:vanishing}
Let $X \subset \PP^n$ be a general hypersurface of degree $d \ge 4$, where $n\ge3$. Then for every smooth hyperplane section $Y = X \cap \{L=0\}$, we have 
\[
\HH^0(Y, T_X|_Y) = 0.
\]
\end{theorem}

\begin{proof}
Consider the following subset of $\PP(V_{n+1,d})\times(\PP^n)^\vee$
\begin{align*}
 \Rc:=\Biggl\{([f],[L])\mid {}&  X=\{f=0\} \ \text{and} \ 
 Y=X\cap\{L=0\}\ \\ & \text{are smooth and }
 \HH^0(Y,T_X|_Y)\neq0\Biggr\}.
\end{align*}
If $\Rc=\varnothing$, there is nothing to prove.
Fix $[L]=[x_n]$ and consider the first-order restriction map
\[
\begin{tikzcd}[row sep=small]
V_{n+1,d} \arrow[r, "\rho_L"] & V_{n,d}\times V_{n,d-1},\\
f=g+x_nh+x_n^2q \arrow[r, mapsto] & (g,h),
\end{tikzcd}
\]
where $g=f|_{x_n=0}$ and $h=\partial_nf|_{x_n=0}$. 
\Cref{lem:criterion} says that
\[
 \Rc_{[L]}
 =\bigl\{[f]\in\PP(V_{n+1,d})\mid
  f \ \text{and} \ g\  \text{are smooth and} \ \rho_L(f)\in B_{n,d}\bigr\}.
\]
The map $\rho_L$ is surjective with kernel $x_n^2V_{n+1,d-2}$. Since $\Rc_{[L]} \subset \{[f]:\rho_L(f) \in B_{n,d}\}$,  \Cref{prop:strata-two-parts}
implies that 
\[
\dim\Rc_{[L]}<\dim\PP(V_{n+1,d})-n.
\]
The action of $\mathrm{PGL}_{n+1}$ preserves $\Rc$ and is transitive on
$(\PP^n)^\vee$, so all fibres of $\pr_2\colon\Rc\to(\PP^n)^\vee$ are
isomorphic. We thus have
\[
 \dim\Rc=n+\dim\Rc_{[L]}<\dim\PP(V_{n+1,d})
\]
and so $\overline{\pr_1(\Rc)}$ is a proper closed subset of
$\PP(V_{n+1,d})$, and every smooth hyperplane section of a general
hypersurface satisfies $\HH^0(Y,T_X|_Y)=0$.
\end{proof}

\begin{remark}
\label{remarkunramifiedness}
When $n \geq 3$, $d \geq 3$, and $d \geq 4$ if $n=3$, the space $\HH^0(Y,T_X|_Y)$ considered in the previous theorem has a classical interpretation in terms of the kernel of the differential of the moduli map
\[
|\mathcal{O}_X(1)|_{\sm} \rightarrow \mathcal{M},
\]
where $\mathcal{M}$ is the moduli space parametrising isomorphism classes of hypersurfaces of degree $d$ in $\PP^{n-1}$ (see, for example, \cite[Section~1]{BE}). Hence the result implies that for a general hypersurface $X \subset \mathbb{P}^n$ the associated moduli map is unramified at every point of $|\mathcal{O}_X(1)|_{\sm}$.
\end{remark}

\begin{remark}
This is closely related to Beauville's result on cubic threefolds \cite{BeauvilleCubic}. For every smooth cubic threefold $X\subset\PP^4$, he proves that the hyperplane-section map
\[
\begin{tikzcd}[row sep=small]
(\PP^4)^\vee\smallsetminus X^\vee \arrow[r, "s_X"]
& \Mmod_{\mathrm{cub}},\\
{[L]} \arrow[r, mapsto] & {[X\cap\{L=0\}]},
\end{tikzcd}
\]
where $\Mmod_{\mathrm{cub}}$ denotes the moduli space of smooth cubic surfaces, is dominant and \'etale at a general $[L]$. More precisely, $s_X$ is \'etale at $[L]$ if and only if $\cdot L\colon R_{F,2}\to R_{F,3}$ is injective, where $R_F$ is the Jacobian ring of $X$. Both spaces have dimension $10$, and the weak Lefschetz property used in \cite{BeauvilleCubic} shows that the resulting determinant is a nonzero form of degree $10$ on $(\PP^4)^\vee$. Its zero divisor cannot be supported on the dual hypersurface $X^\vee$, which has degree $24$. Hence every smooth cubic threefold has some smooth hyperplane section at which $s_X$ ramifies. Thus generic vanishing holds for cubic threefolds, but the vanishing for \emph{every} smooth hyperplane section fails.

The degree bound is therefore sharp when $n\in\{3,4\}$. The argument above makes no claim about cubic hypersurfaces for $n\ge5$; Beauville likewise notes that the required weak Lefschetz statement remains open in those dimensions \cite[Remark~1]{BeauvilleCubic}. Nevertheless, generic finiteness for a general smooth cubic in these dimensions has since been proved by a different argument \cite[Corollary~6.6]{IlardiNasrollahTafazolian}.

The corresponding vanishing theorem already fails for cubic surfaces: if $X\subset\PP^3$ is a smooth cubic surface and $Y$ is a smooth hyperplane section, then $Y$ is elliptic and $\deg(T_X|_Y)=3$, so Riemann--Roch gives $h^0(Y,T_X|_Y)\ge\chi(T_X|_Y)=3$.
\end{remark}

\begin{example}\label{ex:quarticexamples}
The generality hypothesis in \Cref{thm:vanishing} is necessary, as the following two smooth quartic K3 surfaces $S\subset\PP^3$ of high Picard rank illustrate.

(1) Let $S = \{x_0^4+x_1^4+x_2^4+x_3^4=0\}$ be the Fermat quartic. For the coordinate hyperplane section $C = S \cap \{x_3=0\} = \{x_0^4+x_1^4+x_2^4=0\}$, the equation is $g = x_0^4+x_1^4+x_2^4$, and the transverse derivative is $h = \partial_3 f|_{x_3=0} = 0$. The multiplication map $\cdot \bar{h} \colon R_{g,1} \to R_{g,4}$ is therefore the zero map, which has rank $0$. Consequently, $\dim \HH^0(C, T_S|_C) = 3 - 0 = 3$. This is the maximum possible dimension, corresponding geometrically to the existence of an isotrivial pencil of hyperplane sections $x_3 = t x_i$ for $i=0,1,2$ (see \cite[Examples 3.5 and 6.20]{Bertsch}). By contrast, for a general hyperplane section of $S$ (e.g.,\ $x_0+x_1+x_2+x_3=0$), one can check that $\dim \HH^0(C, T_S|_C) = 0$ (see \cite[Theorem 6.21]{Bertsch}).

(2) We note a similar example achieving $\dim \HH^0(C, T_S|_C)=1$. Let $f=x_0^4+x_1^4+x_2^4+x_0x_2^2x_3+x_3^4$ and $S_0=\{f=0\}\subset\PP^3$. Then $S_0$ is a smooth quartic K3, and $C_0=S_0\cap\{x_3=0\}=\{x_0^4+x_1^4+x_2^4=0\}$ is the Fermat plane quartic. In this case $g=x_0^4+x_1^4+x_2^4$, with $\Jf=(x_0^3,x_1^3,x_2^3)$, and $h=x_0x_2^2$. Taking $\lambda=x_2$ gives $\lambda h=x_0x_2^3\in(x_2^3)\subseteq\Jf$, so $\HH^0(C_0,T_{S_0}|_{C_0})\neq0$. The map $\cdot\bar h\colon R_{g,1}\to R_{g,4}$ sends 
\[
\begin{tikzcd}[row sep=small]
x_0 \arrow[r, mapsto] & \overline{x_0^2x_2^2},\\
x_1 \arrow[r, mapsto] & \overline{x_0x_1x_2^2},\\
x_2 \arrow[r, mapsto] & 0.
\end{tikzcd}
\]
The two nonzero images are distinct basis monomials, so $\rank(\cdot\bar h)=2$ and $\dim \HH^0(C_0,T_{S_0}|_{C_0})=3-2=1$ and so $[C_0]$ is a ramification point of the moduli map.
\end{example}

\section{Low degree K3s}\label{sec:low-degree-k3}

We now turn our focus to primitively polarised K3 surfaces of degree $2$, where the geometry of the sextic comes into play. We then apply the preceding results and complementary geometric methods to K3 surfaces of degree $4$.

\subsection{Degree two K3s}

\subsubsection{Finiteness}

In this subsection $(S,H)$ is a principally polarised K3 surface of degree $2$ and Picard rank one, i.e.,\ $H^2=2$ and $\Pic S\cong\ZZ[H]$. Such a surface is a double cover
\[
\begin{tikzcd}
S \arrow[r, "f"] & \PP^2
\end{tikzcd}
\]
ramified along a plane sextic curve $B\subset\PP^2$, with $f^*\OO_{\PP^2}(1)=H$. By the general theory of cyclic covering, we know that the sextic is smooth if and only if $S$ is.

We show first that the moduli map $\mu\colon|H|_{\sm}\to\Mmod_2$ is quasi-finite on the smooth locus. Quasi-finiteness of $\mu$ for the primitive polarisation of a non-uniruled surface of Picard rank one is established in recent work of Chen--Gounelas \cite{CG}; here we give a geometric proof in the degree-two case.

Since $f_*\OO_S\cong\OO_{\PP^2}\oplus\OO_{\PP^2}(-3)$, we have $\HH^0(S,H)\cong\HH^0(\PP^2,\OO_{\PP^2}(1))$, so that every curve in $|H|$ is the pullback $f^*\ell$ of a line $\ell\subset\PP^2$; thus $|H|\cong\Pdual$. For a line $\ell$ transverse to $B$, the curve $f^*\ell$ is the double cover of $\ell\cong\PP^1$ branched at the six points $\ell\cap B$, a smooth curve of genus $2$.

A line $\ell$ meets $B$ at a point with multiplicity $m\ge2$ if and only if $f^*\ell$ has a singularity above that point: in local analytic coordinates, if $y$ defines the sextic near a point $p$ and $\ell$ meets $B$ there with multiplicity $m$, then $f^*\ell$ is given by $t^2-x^m$ near the preimage of $p$, an $A_{m-1}$ singularity. In particular $f^*\ell$ is smooth if and only if $\ell$ is transverse to $B$.

Recall the following theorem (see also \cite[Theorem 10.2]{casalainamartin}).

\begin{theorem}[{\cite[\S6.2.1--6.2.2]{hassett00}, \cite[Example 4.1]{cml12}}]\label{thm:hassett}
    The stable reduction of a curve with an $A_k$ singularity is the transverse union of the normalisation of the curve with a hyperelliptic tail of genus $\lfloor k/2\rfloor$, meeting in one or two points depending on whether $k$ is even or odd.
\end{theorem}

Consider a $1$-dimensional irreducible closed subvariety $T\subset|H|\cong\Pdual$ meeting the smooth locus, and suppose the family it parametrises were isotrivial, so that $\mu|_T$ is constant with value a fixed smooth genus-$2$ curve, an interior point of $\overline{\Mmod}_2$. Being a complete curve in the dual plane $\Pdual$, $T$ meets the dual sextic $B^\vee\subset\Pdual$; the points of $B^\vee$ are the lines tangent to $B$, so $T$ contains such a line $\ell$, for which $f^*\ell$ has an $A_k$ singularity with $k\ge1$. By \Cref{thm:hassett} the stable reduction of the family $\{f^*\ell\}_{\ell\in T}$ at this member is an irreducible nodal curve ($k=1$) or a reducible one ($k\ge2$), hence a point of the boundary $\partial\overline{\Mmod}_2$, contradicting isotriviality. Therefore $\mu$ has no positive-dimensional fibre meeting $|H|_{\sm}$, i.e.,\ it is quasi-finite there.

This establishes an alternative proof of the following result.

\begin{corollary}[{Chen--Gounelas \cite{CG}}]\label{thm:finiteness_h2}
Let $(S,H)$ be a K3 surface of Picard rank one and degree $2$. Then the moduli map $\mu \colon |H|_{\sm} \to \Mmod_2$ is quasi-finite.
\end{corollary}

\subsubsection{Ramification}

Throughout this subsection $g=2$, $H^2=2$: $f\colon S\to\PP^2$ is the double cover branched over a smooth sextic $B=\{F_6=0\}$, $H=f^*\Oc_{\PP^2}(1)$, and for a line $\ell\subset\PP^2$ we write $C_\ell:=f^{-1}(\ell)\in|H|$. Set
\[
\begin{aligned}
   F&:=T_{\PP^2}(-\log B)=\Om_{\PP^2}(\log B)(-3), 
\end{aligned}
\]
so that $c_1(F)=-3h$, and $c_2(F)=21h^2$ for $h:=c_1(\OO_{\PP^2}(1))$.
By the Grauert--M\"ulich theorem, the restriction of $F$ to a general line $\ell$ splits as $\Oc_\ell(-1)\oplus\Oc_\ell(-2)$, so $h^0(\ell,F|_\ell)=0$. A line $\ell$ is called a \emph{jumping line} for $F$ if the splitting is more unbalanced, which is equivalent to $h^0(\ell,F|_\ell)\ge1$.

The following also appears in \cite[Proposition~5.4]{DH}.

\begin{lemma}\label{lem:double-plane}
The curve $C_\ell$ is smooth if and only if $\ell$ is transverse to $B$.
For every line $\ell$, 
\[ 
   \HH^0\!\bigl(C_\ell,\,T_S|_{C_\ell}\bigr)\cong \HH^0\!\bigl(\ell,\,F|_\ell\bigr).
\]
\end{lemma}

\begin{proof}
The Esnault--Viehweg splitting of the pushforward of the cotangent sheaf of a double plane \cite[Lemma~3.16(d)]{EV} (the $\pm1$-eigensheaves of the Galois involution) reads
\[
   f_*\Om_S\ \cong\ \Om_{\PP^2}\ \oplus\ \Om_{\PP^2}(\log B)(-3)\ =\ \Om_{\PP^2}\oplus F.
\]
Since $f$ is finite flat and $C_\ell=f^*\ell$, the projection formula gives
\[
   f_*(\Om_S|_{C_\ell})\ =\ (f_*\Om_S)|_\ell\ =\ \Om_{\PP^2}|_\ell \oplus F|_\ell,
\]
and the definition of direct image gives
\[
   \HH^0(C_\ell,\Om_S|_{C_\ell})\ =\ \HH^0(\ell,\Om_{\PP^2}|_\ell) \oplus \HH^0(\ell,F|_\ell).
\]
For every line $\Om_{\PP^2}|_\ell\cong\Oc_\ell(-1)\oplus\Oc_\ell(-2)$ has no sections, so using $T_S\cong\Om_S$ we get
\[
   \HH^0(C_\ell,T_S|_{C_\ell})\ =\ \HH^0(\ell,F|_\ell).
\] Finally $C_\ell\to\ell$ is the double cover branched on $\ell\cap B$, smooth if and only if that divisor is reduced, i.e.,\ if and only if $\ell$ is transverse.
\end{proof}

\begin{lemma}
\label{lem:F-stable}
The bundle $F$ admits the following resolution by locally free sheaves
\begin{equation}\label{sequence:F}
\begin{tikzcd}[column sep=small]
0 \arrow[r] & \OO_{\PP^2}(-9) \arrow[r, "\psi"]
& \OO_{\PP^2}(-4)^{\oplus3} \arrow[r] & F \arrow[r] & 0.
\end{tikzcd}
\end{equation}
where the map $\psi$ is given by multiplication by the column vector of partial derivatives of $F_6$. In particular, $F$ is stable.
\end{lemma}
\begin{proof}
    Consider the commutative diagram with the Euler sequence
\[
\xymatrix{
0 \ar[r] & \Omega^1_{\PP^2} \ar[r] \ar@{=}[d] & \OO_{\PP^2}(-1)^{\oplus 3} \ar[r]^-{p_A} \ar[d]_{\alpha} & \OO_{\PP^2} \ar[r] \ar[d] & 0 \\
0 \ar[r] & \Omega^1_{\PP^2} \ar[r] & \Omega^1_{\PP^2}(\log B) \ar[r]^-{\mathrm{Res}} & j_*\OO_B \ar[r] & 0.
}
\]
As the kernel of the vertical map on the right is $\OO_{\PP^2}(-6)$, the snake lemma yields a short exact sequence
\[
\begin{tikzcd}[column sep=small]
0 \arrow[r] & \OO_{\PP^2}(-6) \arrow[r, "\psi'"]
& \OO_{\PP^2}(-1)^{\oplus3} \arrow[r]
& \Omega^1_{\PP^2}(\log B) \arrow[r] & 0.
\end{tikzcd}
\]
That the map $\psi'$ is given by multiplication by the column vector of partial derivatives of $F_6$ follows from the right-hand commutative diagram above as well as Euler's homogeneous function theorem (using that the map $p_A$ is multiplication by $(x_0,x_1,x_2)$).
Tensoring this sequence by $\OO_{\PP^2}(-3)$ gives the required exact sequence for $F = \Omega^1_{\PP^2}(\log B)(-3)$.

To prove that $F$ is $\mu$-stable, since its slope is $-\frac{3}{2}$, it suffices to show that there is no line subbundle of degree $\geq -1$. But from \eqref{sequence:F} we see that $\HH^0(F(1))=0$, which concludes the proof.
\end{proof}

\begin{lemma}\label{lem:jump}
Let $\ell$ be transverse to $B$, $\ell\cap B=\{p_1,\dots,p_6\}$. Consider the map
\[
\begin{tikzcd}[row sep=small]
\HH^0(\ell,T_{\PP^2}|_\ell)\cong\CC^5
\arrow[r, "\rho_\ell"] & \CC^6,\\
v \arrow[r, mapsto] & \bigl(v(F_6)(p_i)\bigr)_{i=1}^6.
\end{tikzcd}
\]
Then
\[
\HH^0(\ell,F|_\ell)\cong\ker(\rho_\ell),
\]
and hence
\[
   h^0(\ell,F|_\ell)=5-\rank\rho_\ell.
\]
A transverse jumping line for $F$ yields a smooth $C_\ell$ with $\HH^0(C_\ell,T_S|_{C_\ell})\neq0$.
\end{lemma}

\begin{proof}
The bundle $F$ is defined as the kernel of the logarithmic residue map
\[
\begin{tikzcd}[column sep=small, row sep=small]
0 \arrow[r] & F \arrow[r] & T_{\PP^2} \arrow[r, "\varrho"]
& N_B=\Oc_B(6) \arrow[r] & 0.
\end{tikzcd}
\]
Here $\varrho(v)=v(F_6)|_B$.
For $\ell$ transverse, $\ell\cap B$ is reduced of length $6$, so $\Tor_1(N_B,\Oc_\ell)=0$ and restriction stays exact
\[
\begin{tikzcd}[column sep=small]
0 \arrow[r] & F|_\ell \arrow[r] & T_{\PP^2}|_\ell
\arrow[r, "\varrho_\ell"] & \displaystyle\bigoplus_{i=1}^6\CC
\arrow[r] & 0.
\end{tikzcd}
\]
Taking $\HH^0$ and noting $\varrho_\ell$ induces $\rho_\ell(v)=(v(F_6)(p_i))_i$ gives
\[
   \HH^0(\ell,F|_\ell) = \ker\rho_\ell.
\]
Here $T_{\PP^2}|_\ell\cong\Oc_\ell(2)\oplus\Oc_\ell(1)$, so $h^0=5$, and the map $\HH^0(\PP^2,T_{\PP^2})\to \HH^0(\ell,T_{\PP^2}|_\ell)$ is onto (its cokernel lives in $\HH^1(\PP^2,T_{\PP^2}(-1))=0$). Thus every section is some $v_A|_\ell$, where, via the Euler sequence, $v_A$ is represented by a triple $A=(A_0,A_1,A_2)$ of linear forms, unique modulo a scalar multiple of $(x_0,x_1,x_2)$; equivalently, $A$ is a $3\times3$ matrix modulo scalar matrices. Then
\[
   v_A(F_6)(p_i) = (Ap_i)\cdot\nabla F_6(p_i).
\] The last assertion combines this with \Cref{lem:double-plane}.
\end{proof}

We recall the notation for jumping lines used in \cite{DS}. Let
\[
   V_3:=\CC[x_0,x_1,x_2],
\]
let $D=\{G=0\}\subset\PP^2$ be a reduced plane curve of degree $d$, and set
\[
   E_D:=T_{\PP^2}(-\log D)(-1).
\]
For a line $\ell\subset\PP^2$, its ordered splitting type is the unique pair
$d_1^\ell\le d_2^\ell$ such that
\[
   E_D|_\ell\cong\OO_\ell(-d_1^\ell)\oplus\OO_\ell(-d_2^\ell).
\]
If $\ell_{\mathrm{gen}}$ is a general line, the jumping order of $\ell$ is
$o(\ell):=d_1^{\ell_{\mathrm{gen}}}-d_1^\ell$, and $\ell$ is a jumping line
when $o(\ell)>0$. Write
\[
\begin{aligned}
   AR(G)&:=\left\{(a_0,a_1,a_2)\in V_3^3:
      \sum_{i=0}^2a_i\partial_iG=0\right\},\\
   \operatorname{mdr}(G)&:=\min\{m\ge0:AR(G)_m\ne0\},\\
   N(G)&:=\widehat{J_G}/J_G,
\end{aligned}
\]
where $J_G=(\partial_0G,\partial_1G,\partial_2G)$ and $\widehat{J_G}$ is its
saturation with respect to $(x_0,x_1,x_2)$. If $D$ is smooth, then
$\widehat{J_G}=V_3$, so $N(G)=V_3/J_G$. For a line
$\ell=\{\alpha_\ell=0\}$, multiplication by $\alpha_\ell$ gives
\[
\begin{tikzcd}
N(G)_{k+d-2} \arrow[r, "\cdot\alpha_\ell"]
& N(G)_{k+d-1}.
\end{tikzcd}
\]
Set
\[
k(G,\ell):=\min\{k\ge0:
\cdot\alpha_\ell\colon N(G)_{k+d-2}\to N(G)_{k+d-1}
\text{ is not injective}\}.
\]
Then \cite[Proposition~4.1]{DS} gives
\[
   d_1^\ell=\min\{\operatorname{mdr}(G),k(G,\ell)\}.
\]
Accordingly, for $k\ge0$ the $k$-th jumping locus is
\[
   V_k(D):=\{\ell\in\Pdual:d_1^\ell\le k\}.
\]

\begin{proposition}\label{prop:reduced-example_h2}
Let $B_1\subset\PP^2$ be the plane sextic defined over $\QQ$ by
\[
 F_1=X^6+Y^6+Z^6+X^5Y+Y^5Z+Z^5X
     +2X^3Y^2Z+3X^2YZ^3+5XY^3Z^2.
\]
Then $B_1$ is smooth, and its jumping-line scheme is reduced of length
$171$ and is disjoint from the dual curve $B_1^\vee$. In particular, all
its jumping lines are transverse to $B_1$.
\end{proposition}

\begin{proof}
We apply the preceding notation with $D=B_1$, $G=F_1$, and $d=6$.
A direct calculation shows that $B_1$ is smooth.  Consequently,
$\operatorname{mdr}(F_1)=5$ and $AR(F_1)$ is generated by the three Koszul
syzygies of degree $5$ \cite[Example~6.1]{DS}.  The ordered splitting type of
$E_{B_1}=F_{B_1}(-1)$ on a general line is $(2,3)$; equivalently,
\[
   F_{B_1}|_\ell
   \cong\OO_\ell(1-d_1^\ell)\oplus\OO_\ell(1-d_2^\ell).
\]
It follows that $h^0(\ell,F_{B_1}|_\ell)>0$ if and only if
$d_1^\ell\le1$, so the jumping scheme is $V_1(B_1)$.

Writing $n(F_1)_m:=\dim N(F_1)_m$, the Hilbert function of the Jacobian
algebra of a smooth sextic gives
\[
   n(F_1)_4=15,\qquad n(F_1)_5=18,\qquad n(F_1)_6=19.
\]
By \cite[Proposition~4.1]{DS}, a line $\ell$ belongs to $V_1(B_1)$ precisely
when at least one of the multiplication maps
\[
\begin{tikzcd}
N(F_1)_4 \arrow[r, "\cdot\alpha_\ell"] & N(F_1)_5
\end{tikzcd}
\qquad\text{or}\qquad
\begin{tikzcd}
N(F_1)_5 \arrow[r, "\cdot\alpha_\ell"] & N(F_1)_6
\end{tikzcd}
\]
fails to be injective.  The socle degree is $3(d-2)=12$, and
\cite[Remark~4.3]{DS} shows that noninjectivity of the first map implies
noninjectivity of the second.  Thus $V_1(B_1)$ is the degeneracy locus of
the second map.

More concretely, a nonzero class $[A]\in N(F_1)_5$ in its kernel satisfies
\[
   A\alpha_\ell=\sum_{i=0}^2 B_i\partial_iF_1\in(J_{F_1})_6
\]
for linear forms $B_i$.  Restriction to
$\ell=\{\alpha_\ell=0\}$ gives a nontrivial relation among the restrictions
of the partial derivatives of $F_1$.  Hence the same degeneracy locus is
defined by failure of injectivity of
\[
\begin{tikzcd}
\HH^0\bigl(\ell,\OO_\ell(1)^3\bigr)
\arrow[r, "\nabla F_1"]
& \HH^0\bigl(\ell,\OO_\ell(6)\bigr).
\end{tikzcd}
\]
After parametrising a line in an affine chart of $\Pdual$, this is a
$7\times6$ matrix, and its seven maximal minors define the jumping scheme.

The final exact Macaulay2 computation in
Appendix~\ref{app:reduced-jumping} treats all three standard affine charts.
It shows that the complete projective degeneracy scheme is contained in the
first chart, where it has dimension $0$ and degree $171$.  Adjoining the
$2\times2$ minors of the Jacobian matrix gives the unit ideal; the Jacobian
criterion therefore shows that this zero-dimensional scheme is reduced.
Finally, adjoining the degree-$30$ binary resultant defining the dual curve
$B_1^\vee$ also gives the unit ideal.  Thus none of the $171$ jumping lines
is tangent to $B_1$.
\end{proof}

\begin{remark}\label{rem:nonreduced-example_h2}
The smooth sextic $B_0=\{F_0=0\}$, where
\[
 F_0=\prod_{i=0}^{5}(x_0-i\,x_1)
 +x_2^{2}\bigl(x_0^4+x_1^4+x_2^4
 +x_0x_1x_2(x_0+x_1+x_2)\bigr),
\]
also exhibits the number $171$: its jumping scheme is zero-dimensional of
degree $171$.  In this example, however, reducedness fails.  Indeed, the transverse line
$\ell_0=\{x_2=0\}$ is an explicit jumping line in the deeper determinantal stratum.
A direct calculation gives
\[
 h^0\bigl(\ell_0,F_{B_0}|_{\ell_0}\bigr)=2,
\]
or equivalently, the corresponding $7\times6$ restriction matrix has rank
$4$.  Hence $\ell_0$ lies in the singular locus of the vanishing of the minors.
Since that scheme is zero-dimensional, it is nonreduced at $\ell_0$.
\end{remark}

\begin{theorem}\label{thm:very-general_h2}
For a general K3 surface $(S,H)$ of degree $H^2=2$, the ramification locus 
\[ 
   R_S=\{[C]\in|H|_{\sm}:\HH^0(C,T_S|_C)\neq0\}
\] 
consists of exactly $171$ distinct points, all corresponding to jumping lines for the bundle $F$ of splitting type $\Oc_\ell\oplus\Oc_\ell(-3)$. In particular, the moduli map $\mu\colon|H|_{\sm}\to \Mmod_2$ is ramified.
\end{theorem}

\begin{proof}
Let $U\subset\PP^{27}$ be the locus of smooth sextics. By
\Cref{lem:double-plane}, the ramification locus $R_{S_B}$ is identified
with the transverse jumping lines of
$F_B=T_{\PP^2}(-\log B)$.  Let
\[
   \mathscr J\subset U\times\Pdual
\]
be the determinantal scheme on which the universal version of the map
$\rho$ in \Cref{lem:jump} has rank at most $4$, and let
$\pi_U\colon\mathscr J\to U$ be the first projection.  It has expected
codimension $(5-4)(6-4)=2$.

By \Cref{prop:reduced-example_h2}, the fibre over $B_1$ is finite,
reduced, and disjoint from the relative tangency divisor.  The
quasi-finite locus of the proper morphism $\pi_U$ is open; removing the proper
image of its complement, we may therefore shrink $U$ around $B_1$ so that
$\pi_U$ is finite.  Along the fibre over $B_1$ the maximal-minor ideal has
height $2$, and hence is Cohen--Macaulay by the Eagon--Northcott theorem.
Since $U$ is smooth, miracle flatness shows, after shrinking once more,
that $\pi_U$ is finite and flat.  Its fibre over $B_1$ is reduced, hence
\'{e}tale, so openness of the \'{e}tale locus and properness allow us to
shrink to a nonempty open subset $U_0\subset U$ over which $\pi_U$ is finite
\'{e}tale.  We may also remove the proper image of the intersection with
the relative tangency divisor.  Consequently, for every $B\in U_0$, the
jumping scheme consists of $171$ distinct lines.

Finally, a jumping line $\ell$ has splitting type $\Oc_\ell(k)\oplus\Oc_\ell(-3-k)$ with $k\ge1$ (equivalently $h^0(\ell,F_B|_\ell)\ge2$) exactly when $\rank\rho_\ell\le3$. The locus $\{\rank\rho\le3\}$ is contained in the singular locus of the jumping scheme $\{\rank\rho\le4\}$ \cite[Ch.~II, \S2]{ACGH}, which is empty for a general sextic since we proved above it is reduced. Hence every jumping line has $h^0(\ell,F_B|_\ell)=1$ and splitting type exactly $\Oc_\ell\oplus\Oc_\ell(-3)$.
\end{proof}

It would be interesting to know if these $171$ jumping lines have any interpretation in terms of classical invariants of sextics, e.g.,\ the $324$ bitangents or the $72$ flex lines.

\begin{remark}
\label{Thom-Porteous}
As a sanity check, in this remark we give another proof that the zero-dimensional scheme of jumping lines has length $171$, assuming only that it is finite for a special sextic, possibly nonreduced. By our assumption there is a dense open subset $U_0\subseteq U$ such that for $B\in U_0$, the ramification locus $R_{S_B}$ is finite and nonempty. For a general K3 surface, the corresponding sextic lies in $U_0$, meaning $R_S$ is a zero-dimensional scheme. Because the jumping locus has the expected codimension $2 = \dim\Pdual$ for $B \in U_0$, its length is exactly the Thom--Porteous class $\int_{\Pdual} c_2(\mathcal N-\mathcal T)$, where
$\mathcal T = q_*p^*T_{\PP^2}$ and $\mathcal N = q_*p^*N_B$ are
pushforwards from the point-line incidence variety
$\mathcal I \subset \PP^2 \times \Pdual$ with projections $p$ and $q$.
The universal line $\mathcal I\subset\PP^2\times\Pdual$ is a divisor of
class $\OO(1,1)$, giving the resolution
\[
\begin{tikzcd}[column sep=small]
0 \arrow[r] & \OO_{\PP^2\times\Pdual}(-1,-1) \arrow[r]
& \OO_{\PP^2\times\Pdual} \arrow[r]
& \OO_{\mathcal I} \arrow[r] & 0.
\end{tikzcd}
\]
For $E=T_{\PP^2}$, tensoring this resolution with $p^*E$ remains exact because $E$ is locally free. For $E=N_B=\OO_B(6)$, it remains exact because the universal incidence equation is a non-zero-divisor on $B\times\Pdual$, or equivalently because the relevant $\Tor_1$ vanishes. In either case, taking the K-theoretic pushforward $q_!$ gives
\[
   [q_*(p^*E|_{\mathcal I})]
   =h^0(E)[\OO_{\Pdual}]
   -h^0(E(-1))[\OO_{\Pdual}(-1)],
\]
provided the higher cohomology of $E$ and $E(-1)$ vanishes.

For $E = T_{\PP^2}$, we have $h^0(T_{\PP^2})=8$ and $h^0(T_{\PP^2}(-1))=3$. This gives $[\mathcal T] = 8[\OO] - 3[\OO(-1)]$, so its total Chern class is $c_t(\mathcal T) = (1-h)^{-3} = 1 + 3h + 6h^2$, where $h=c_1(\OO_{\Pdual}(1))$. For $E = N_B = \OO_B(6)$, we use the exact sequence $0 \to \OO_{\PP^2} \to \OO_{\PP^2}(6) \to \OO_B(6) \to 0$ to find $h^0(\OO_B(6))=27$ and $h^0(\OO_B(5)) = 21$. This yields $[\mathcal N] = 27[\OO] - 21[\OO(-1)]$, so $c_t(\mathcal N) = (1-h)^{-21} = 1 + 21h + 231h^2$. The total Chern class of the virtual bundle $\mathcal N-\mathcal T$ is therefore
\[
   c_t(\mathcal N-\mathcal T)\ =\ \frac{1+21h+231h^2}{1+3h+6h^2}\ =\ 1+18h+171h^2.
\]
The degree-$2$ component is $171$, meaning the jumping scheme has length $171$. However, to deduce that it is reduced, that is, that the jumping lines are all distinct, one must further use \Cref{prop:reduced-example_h2} and the proof of \Cref{thm:very-general_h2}.
\end{remark}

\begin{remark}\label{rem:degree2_vectorbundle}
   The stable rank $2$ bundle $E = F(3)$ of odd degree $c_1(E)=3$ can be normalised to $E_{\mathrm{norm}} = E(-2) = F(1)$, which has $c_1 = -1$ and $c_2 = 19$. By results of Hulek \cite{Hulek} (see also \cite[p.\,119]{OSS} and \cite{DK1, DK2}), the locus of lines of the second kind forms a curve of degree $2(c_2(E_{\mathrm{norm}})-1) = 36$ in the dual plane $\Pdual$. The $171$ jumping lines found in \Cref{thm:very-general_h2} correspond geometrically to the singularities of this degree-$36$ curve.
\end{remark}

\subsection{Degree four K3s}

\subsubsection{Finiteness}

In this subsection $(S,H)$ is a K3 surface with $H^2=4$ and $\Pic S\cong\ZZ[H]$; equivalently $S\subset\PP^3$ is a smooth quartic surface with $H=\OO_S(1)$. Quasi-finiteness of the moduli map $\mu\colon|H|_{\sm}\to\Mmod_3$ is established in \cite{CG}. We record here a partial geometric argument for the primitive system $|H|$ as it highlights the difficulty of the problem.

Consider a generically smooth, isotrivial family of hyperplane sections $T\subset|H|$. Under the isomorphism $\HH^0(S,H)\cong\HH^0(\PP^3,\OO_{\PP^3}(1))$ it is a family of planes cutting $S$. Over the normalisation of $T$ (a smooth curve) the universal $\PP^2$-bundle of ambient planes has trivial Brauer class by Tsen's theorem, hence is Zariski-locally the projectivisation of a rank-$3$ bundle; so in a neighbourhood $U\subset T$ of a point parametrising a singular section we may regard $U$ as a family of quartic curves in a \emph{fixed} plane $\PP^2$. By the following well-known lemma, every smooth member of this family is a $\mathrm{PGL}_3$-translate of one fixed curve $C\subset\PP^2$.

\begin{lemma}\label{lem:quartic-orbit}
    Let $C,D\in|\OO_{\PP^2}(d)|$ be smooth plane curves of degree $d\ge4$. If $C\cong D$, then there is $g\in\mathrm{PGL}_3$ with $g(C)=D$.
\end{lemma}
\begin{proof}
    By \cite[p.~56, Exercise~18]{ACGH}, $\OO_C(1)$ is the unique $g^2_d$ on $C$. Hence if $\phi\colon C\to D$ is an isomorphism then $\phi^*\OO_D(1)\cong\OO_C(1)$, and the claim follows since $\HH^0(C,\OO_C(1))\cong\HH^0(\PP^2,\OO_{\PP^2}(1))$.
\end{proof}

Thus $U_{\sm}$ lies in the $\mathrm{PGL}_3$-orbit of a fixed smooth plane quartic $C$, and $U$ lies in the closure of that orbit inside $|\OO_{\PP^2}(4)|$. To rule out such a family $T$ on a quartic K3 of Picard rank one, one may therefore study the $\mathrm{PGL}_3$-orbit closures of smooth plane quartics, or equivalently, the limits arising in the closure of $\mathrm{PGL}_3$ inside the space of $3\times3$ matrices. This is the subject of the work of several papers of Aluffi--Faber \cite{aluffifaber0, aluffifaber1,aluffifaber2}.

The relevant orbits are classified in \cite[Section~1]{aluffifaber0} (see the appendix of \cite{aluffifaber1} for a summary): it suffices to treat one-dimensional connected orbits, which are copies of $\mathbb G_m$ or $\mathbb G_a$ in $\mathrm{PGL}_3$. All but one type of limit curve contain a line or a smooth conic, i.e.,\ an irreducible component isomorphic to $\PP^1$, which cannot lie on a K3 of Picard rank one. The sole exception is the type $y^d+\lambda z^a x^{d-a}$ (case 11 of \cite[Appendix]{aluffifaber1}), which for $d=4$ is the higher cusp $y^4=zx^3$; it arises exactly when the smooth quartic whose orbit we consider has a quadriflex line. We did not find a way to rule out this curve from appearing as a limit of an isotrivial family in $|\OO_S(1)|$ on a quartic K3 of Picard rank one; it is, however, excluded as a consequence of \cite{CG}.

\subsubsection{Ramification}

\begin{corollary}\label{cor:quartic-unramified}
Let $S\subset\PP^3$ be a general quartic surface. Then for every smooth hyperplane section $C\in|\Oc_S(1)|$, one has
\[
   \HH^0(C,T_S|_C)=0.
\]
Consequently, the moduli map $\mu\colon|\Oc_S(1)|_{\sm}\to\Mmod_3$ is unramified.
\end{corollary}

\begin{proof}
Apply \Cref{thm:vanishing} with $(n,d)=(3,4)$, and then use \Cref{lem:kernel}.
\end{proof}

\section{Stability and ramification: Examples}\label{sec:stability-ramification}

As already observed in \Cref{sec:high-degree}, the stability of $T_S|_C$ on a K3 surface $S$ implies the unramifiedness of the moduli map $|\OO(C)|_{\sm} \rightarrow \Mmod_g$. Moreover, a careful study of the (semi)stability of $T_S|_C$ has been used in \cite{DH} to prove maximal variation for almost all polarisations $H$. In this section we further investigate how the (semi)stability of $T_S|_C$ interacts with the existence of sections and hence, for smooth $C$, with ramification of the moduli map. More precisely, we show that instability of the restriction is not equivalent to nonvanishing of $\HH^0(C,T_S|_C)$.
 
As for non-semistability, the situation is less clear. Following some
constructions in \cite{GO}, we provide explicit families of smooth
curves \(\{C_t\}\) contained in the ramification locus such that
\(T_S|_{C_t}\) is not semistable. These examples suggest that
non-semistability can indeed give rise to ramification.

\begin{example}
Two trivial examples to keep in mind. If $R\cong\PP^1$ and $R\subset S$, then $T_S|_R\cong\OO(-2)\oplus\OO(2)$, so $\HH^0(R,T_S|_R)=3$, and the restriction is not semistable. Similarly, if $E\subset S$ is a smooth genus-one curve, then the conormal sequence, together with $T_S\cong\Omega_S$, gives $\HH^0(E,T_S|_E)=1$, and the restriction is not stable.
\end{example}

\begin{lemma}\label{lem:nodalh0}
If $C\subset S$ is an integral rational curve whose only singularity is one ordinary node, then $\HH^0(\tsc) \neq 0$.
\end{lemma}
\begin{proof}
Let $\nu\colon S' \ra S$ be the blow-up of $S$ at the node of $C$. Let $C' \in |H-2E|$ be the normalisation. Denote by $p$ and $q$ the intersection points of $E$ with $C'$. Then we have the following exact sequence
\begin{equation}
\label{equation3}
\begin{tikzcd}[column sep=small]
0 \arrow[r] & \nu^*\Omega_S|_{C'} \arrow[r]
& \Omega_{S'}|_{C'} \arrow[r] & \OO_p\oplus\OO_q \arrow[r] & 0.
\end{tikzcd}
\end{equation}
Now observe that $h^0(\Omega_{S'}|_{C'})=h^0(\OO_{C'}(-C'))=5$, since $C^2=0$ ($C$ is a curve of arithmetic genus $1$), and  $C'^2=0-4=-4$ in $S'$, and $C'$ is rational. Hence taking global sections in  \eqref{equation3} we find
\begin{equation*}
5 \geq h^0(\nu^* \Omega_S|_{C'})\geq 3.
\end{equation*}
Hence we get $5 \geq h^0(C',\nu^* \Omega_S|_{C'}) \simeq h^0(C',\nu|_{C'}^*\Omega_S|_C) \simeq h^0(C,\Omega_S|_C \otimes (\nu|_{C'})_*\OO_{C'}) \geq 3$. Denote by $x$ the node. We have
\[
\begin{tikzcd}[column sep=small]
0 \arrow[r] & \Omega_S|_C \arrow[r]
& \Omega_S|_C\otimes(\nu|_{C'})_*\OO_{C'}
\arrow[r] & \OO_x^{\oplus2} \arrow[r] & 0.
\end{tikzcd}
\]
From this and $T_S\cong\Omega_S$ we deduce that $h^0(C,T_S|_C)\geq 1$.
\end{proof}

\subsection{\texorpdfstring{$T_S|_C$ not stable yet $h^0(C,T_S|_C)=0$}{Unstable restriction with no sections}}\label{subsection:notstableh0}

The vanishing of $\HH^0(C,T_S|_C)$ is not equivalent to the stability of the restriction, as the following example shows.

Consider a general polarised $K3$ surface $(S,H)$ of degree $8$. Then by \cite[Theorem A]{GO} there exists a positive-dimensional family of curves in $|3H|$ such that $T_S|_C$ is not stable.  On the other hand, by Theorem \ref{prop:knutsen-normal}, $\HH^0(\tsc)=0$.

\subsection{\texorpdfstring{$T_S|_C$ not semistable and $h^0(C,T_S|_C)=1$ in degree 2}{Non-semistable restriction with one section in degree 2}}\label{sec:stability-degree2}
Let $(S,H)$ be a very general $K3$ surface of degree $2$. Consider curves of the form $C=S_1 \cap S_2$ where $S_1 \in |L+2H|$, and $S_2 \in |L+4H|$ as in \cite[Proposition~4.1]{GO}. These curves are isomorphic to their images in $S$ and form a 15-dimensional family of curves where $T_S|_C$ is not semistable. We have
\[
\begin{tikzcd}[column sep=small]
0 \arrow[r] & \Omega_\pi(L)|_C \arrow[r]
& \pi^*\Omega_S|_C \arrow[r] & L|_C \arrow[r] & 0.
\end{tikzcd}
\]
Taking determinants and using that $L\cdot C<0$, we obtain
$\Omega_{\pi}(L)|_C\simeq\OO_C(-L)$ and
$\HH^0(C,\OO_C(-L))\simeq\HH^0(C,\tsc)$. 
\begin{lemma}\label{lem:h0degree2} We have
\begin{equation*}
h^0(C,\OO_C(-L))=h^1(\PP(\Omega^1_S),6H+L)=h^1(S,\Omega_S(6H))=1.
\end{equation*}
\end{lemma}

\begin{proof}
Since $g(C)=37$ and $\deg(L|_C)=8$, Riemann--Roch gives
\begin{equation}
\label{equation4}
h^0(C,\OO_C(-L))=h^1(C,\OO_C(-L))-28
 =h^0\bigl(C,\OO_C(6H+L)\bigr)-28.
\end{equation}
Since $C= S_1 \cap S_2$, $S_1 \in |L+2H|$,  $S_2 \in |L+4H|$, we have the following short exact sequence
\[
\begin{tikzcd}[column sep=small]
0 \arrow[r] & \OO_{S_1}(2H) \arrow[r]
& \OO_{S_1}(6H+L) \arrow[r]
& \OO_C(6H+L) \arrow[r] & 0.
\end{tikzcd}
\]
Using that $S_1 \rightarrow S $ is a blow-up, we get $h^0(S_1,\OO_{S_1}(2H))=h^0(S,2H)=2+4=6$.  $h^1(S_1,\OO_{S_1}(2H))=h^1(S,2H)=0$. Hence
\begin{equation}
\label{equation2}
h^0(C,6H+L|_C)=h^0(S_1,6H+L)-h^0(S_1,2H)=h^0(S_1,6H+L)-6.
\end{equation}
Now consider the following short exact sequence
\begin{equation}
\label{equation5}
\begin{tikzcd}[column sep=small]
0 \arrow[r] & \OO_{\PP(\Omega_S)}(4H) \arrow[r]
& \OO_{\PP(\Omega_S)}(6H+L) \arrow[r]
& \OO_{S_1}(6H+L) \arrow[r] & 0.
\end{tikzcd}
\end{equation}
Notice that
\begin{align*}
h^0(\OO_{\PP(\Omega_S)}(4H))&=h^0(S,4H)=18,\\
h^1(\OO_{\PP(\Omega_S)}(4H))&=0,
\end{align*}
and
\begin{align*}
h^0(\OO_{\PP(\Omega_S)}(6H+L))
   &=h^0(S,\Omega_S(6H))\\
   &=(6H)^2-20+h^1(S,\Omega_S(6H))\\
   &=52+h^1(S,\Omega_S(6H)).
\end{align*}
The last computation follows from Grothendieck--Riemann--Roch (see \cite[Theorem 3.1]{Totaro}). 
Hence from \eqref{equation2} and \eqref{equation5}  we find that
\begin{align*}
h^0(C,6H+L|_C)
   &=h^0(S_1,6H+L)-6\\
   &=52+h^1(S,\Omega_S(6H))-18-6\\
   &=28+h^1(S,\Omega_S(6H)).
\end{align*}
This, together with \eqref{equation4}, gives
\begin{equation*}
h^0(C,\OO_C(-L))=h^1(S,\Omega_S(6H))
 \bigl(=h^1(\PP(\Omega_S),6H+L)\bigr).
\end{equation*}
By \cite[Lemma 3.7]{DedieuSernesi}, we have that $\HH^1(S,\Omega_S(6H))=1$.
\end{proof}

\begin{remark}
Let $C \in |6H|$ be a smooth curve as above. Then $\mathrm{Cliff}(C) \geq 3$ from \cite[Lemma 8.3]{knutsen}. By \cite[Corollary 8.6]{CDS}, we have that for curves in $|6H|$ the moduli map is finite (not just generically finite) onto its image. However, \Cref{lem:h0degree2} shows that the map is ramified along a locus in $|6H|$ of dimension at least $15$.
\end{remark}

\subsection{\texorpdfstring{$T_S|_C$ not semistable and $h^0(C,T_S|_C)=1$ in degree 4}{Non-semistable restriction with one section in degree 4}}\label{sec:stability-degree4}

Let $S$ be a generic quartic in $\PP^3$. Consider the complete intersection curve $C=S_1 \cap S_2 \subset \PP(\Omega_S)$ constructed in \cite[Proposition~4.3(1)]{GO}. Recall that $S_1, S_2 \in |L+2H|$, and the image of $C$ in $S$ lives in $|4H|$. We have
\[
\begin{tikzcd}[column sep=small]
0 \arrow[r] & \Omega_\pi(L)|_C \arrow[r]
& \pi^*\Omega_S|_C \arrow[r] & L|_C \arrow[r] & 0.
\end{tikzcd}
\]
Taking determinants gives $\Omega_{\pi}(L)|_C\simeq\OO_C(-L)$. In
particular, $\HH^0(C,\OO_C(-L))\simeq\HH^0(C,\tsc)$ because
$L\cdot C<0$. 
\begin{lemma}\label{lem:h0degree4}
In this setting we have
\begin{equation*}
h^0(C,\OO_C(-L))=h^1(\PP(\Omega^1_S),4H+L)=h^1(S,\Omega_S(4H))=1.
\end{equation*}

\end{lemma}

The proof is identical in spirit to that of \Cref{lem:h0degree2}, using \cite[Theorem~3.6]{DedieuSernesi} to conclude that $h^1(S,\Omega_S(4H))=1$. We omit the details.

\section{Cohomology jumping loci}\label{section:jumpingloci}

Let $S$ be a complex projective $K3$ surface and let $H$ be an ample,
base-point-free line bundle on $S$; we do not assume that $H$ is
primitive. Then
\[
|H|\cong\PP^g,\qquad g=1+\frac{H^2}{2}\ge2.
\]
The results of the preceding sections naturally lead to the study of the loci where the dimension of $\HH^0(C,T_S|_C)$ jumps. We first give them a determinantal description, which is valid over the whole linear system, including its singular members.

Set
\[
A_H:=\HH^1(S,T_S(-H)),\qquad B:=\HH^1(S,T_S),\qquad a_H:=\dim A_H.
\]
Multiplication by the universal equation defines a morphism of vector bundles on $|H|$
\begin{equation}\label{eq:universal-cup-product}
\begin{tikzcd}
A_H\otimes\Oc_{|H|}(-1) \arrow[r, "\Phi_H"]
& B\otimes\Oc_{|H|}.
\end{tikzcd}
\end{equation}
This is the global version over $|H|$ of the cup-product map studied in \Cref{lem:kernel}.
More explicitly, if $[C]=[s_C]\in |H|$, then the fibre of \eqref{eq:universal-cup-product} at $[C]$ is 
\[
\begin{tikzcd}[row sep=small]
A_H\otimes\langle s_C\rangle \arrow[r, "(\Phi_H)_{[C]}"] & B,\\
\alpha\otimes s_C \arrow[r, mapsto] & \alpha\cup s_C.
\end{tikzcd}
\]

\begin{proposition}\label{prop:jumping-determinantal}
For every $[C]\in |H|$ there is a canonical identification
\[
\HH^0(C,T_S|_C)\cong\ker(\Phi_H)_{[C]}.
\]
Consequently, for $k\ge1$ the cohomology jumping locus has the natural determinantal scheme structure
\[
S_k(H):=D_{a_H-k}(\Phi_H)
       =\{[C]\in |H|:h^0(C,T_S|_C)\ge k\},
\]
where $S_k(H)=\varnothing$ if $k>a_H$. In particular
\[
|H|\supseteq S_1(H)\supseteq S_2(H)\supseteq\cdots
\]
is a decreasing filtration by closed subschemes.
\end{proposition}

\begin{proof}
Let $p\colon S\times |H|\to S$ and $q\colon S\times |H|\to |H|$ be the projections, and let $\mathcal C\subset S\times |H|$ be the universal divisor. Since
\[
\Oc_{S\times |H|}(-\mathcal C)
\cong p^*\Oc_S(-H)\otimes q^*\Oc_{|H|}(-1),
\]
the universal restriction sequence is
\[
\begin{tikzcd}[column sep=small]
0 \arrow[r]
& p^*T_S(-H)\otimes q^*\Oc_{|H|}(-1)
\arrow[r, "\cdot s_{\mathrm{univ}}"]
& p^*T_S \arrow[r] & p^*T_S|_{\mathcal C} \arrow[r] & 0.
\end{tikzcd}
\]
Applying cohomology and base change to $q$, together with the projection formula, gives
\[
\begin{aligned}
R^1q_*\bigl(p^*T_S(-H)\otimes q^*\Oc_{|H|}(-1)\bigr)
&\cong \HH^1(S,T_S(-H))\otimes\Oc_{|H|}(-1)
 = A_H\otimes\Oc_{|H|}(-1),\\
R^1q_*p^*T_S
&\cong \HH^1(S,T_S)\otimes\Oc_{|H|}
 = B\otimes\Oc_{|H|}.
\end{aligned}
\]
Thus the map on first higher direct images induced by multiplication by the universal section is precisely $\Phi_H$. On the fibre over $[C]$ it is the cohomology map associated with
\[
\begin{tikzcd}[column sep=small]
0 \arrow[r] & T_S(-H) \arrow[r, "\cdot s_C"]
& T_S \arrow[r] & T_S|_C \arrow[r] & 0.
\end{tikzcd}
\]
As $\HH^0(S,T_S)=0$, we conclude that
\[
\HH^0(C,T_S|_C)
\cong\ker\bigl((\Phi_H)_{[C]}\colon A_H\to B\bigr).
\]
The description by minors of $\Phi_H$ now gives the asserted scheme structure and filtration.
\end{proof}

Set
\[
S_k(H_{\sm}):=S_k(H)\cap |H|_{\sm}.
\]
Thus the subscript $\sm$ records that we restrict the linear system to its open locus of smooth members. On $|H|_{\sm}$, these determinantal loci are precisely the corank loci of the differential of the moduli map
\[
\begin{tikzcd}
{|H|_{\sm}} \arrow[r, "\mu"] & \Mmod_g.
\end{tikzcd}
\]
Indeed, the normal bundle sequence gives $\ker d\mu_{[C]}\cong\HH^0(C,T_S|_C)$, as in \Cref{lem:kernel}. Thus $S_1(H_{\sm})$ is the ramification locus of $\mu$, while the higher $S_k(H_{\sm})$ measure the size of its infinitesimal fibres.

\begin{proposition}\label{prop:jumping-codimension}
Let $1\le k\le g$, and assume that $S_k(H_{\sm})$ is nonempty. Then every irreducible component $W\subseteq S_k(H_{\sm})$ satisfies
\[
\codim_{|H|_{\sm}} W\ge k-2.
\]
If the moduli map $\mu$ is quasi-finite onto its image, then
\[
\codim_{|H|_{\sm}} W\ge k.
\]
\end{proposition}

\begin{proof}
For $[C]\in |H|_{\sm}$, the normal bundle sequence gives
\[
\begin{tikzcd}[column sep=small]
0 \arrow[r] & \HH^0(C,T_S|_C) \arrow[r]
& \HH^0(C,N_{C/S}) \arrow[r, "d\mu_{[C]}"]
& \HH^1(C,T_C).
\end{tikzcd}
\]
Since $T_{|H|_{\sm},[C]}\cong\HH^0(C,N_{C/S})$ and $\dim |H|_{\sm}=g$, the differential has rank at most $g-k$ along $W$. In characteristic zero, the generic rank theorem therefore gives
\[
\dim\mu(W)\le g-k.
\]

It remains to bound the fibres. Let $T$ be a positive-dimensional irreducible component of a fibre of $\mu$. After a generically finite base change, the corresponding family of smooth curves becomes birationally trivial, so its evaluation map gives a dominant rational map
\[
C\times T'\dashrightarrow S.
\]
Here dominance follows because a positive-dimensional family of distinct smooth members of $|H|$ sweeps out $S$. Since a $K3$ surface is not uniruled, \cite[Lemma~2.3]{CG} gives $\dim T'=\dim T\le2$. Hence
\[
\dim W\le\dim\mu(W)+2\le g-k+2,
\]
which proves $\codim_{|H|_{\sm}}W\ge k-2$. If $\mu$ is quasi-finite, then $\dim W=\dim\mu(W)$, and the stronger bound follows.
\end{proof}

We now conclude with some speculative remarks connecting the preceding results with the examples in earlier sections.

\begin{remark}[Expected codimension and excess degeneracy]\label{rem:jumping-expected}
The description by $\Phi_H$ suggests a first expected-codimension calculation. Since $\dim B=h^1(S,T_S)=20$, the expected codimension of $S_k(H)=D_{a_H-k}(\Phi_H)$ is
\[
k(20-a_H+k)
\]
whenever $1\le k\le a_H$ and $\Phi_H$ is generically injective. The generic injectivity is exactly maximal variation, which holds for ample systems on $K3$ surfaces. The map $\Phi_H$ is nevertheless very far from a general morphism between bundles - the actual loci may therefore have excess dimension, as the degree-$2$ example shows.

There is a second expected-codimension calculation on $|H|_{\sm}$. Here we regard $\Mmod_g$ as the smooth Deligne--Mumford stack of curves, rather than its coarse moduli space. The dual differential
\[
\begin{tikzcd}
\mu^*\Omega_{\Mmod_g} \arrow[r, "(d\mu)^\vee"] & \Omega_{|H|_{\sm}}
\end{tikzcd}
\]
has source rank $3g-3$ and target rank $g$ and the locus where its rank is $\le g-k$ has expected codimension
\[
k(2g-3+k).
\]
For $k=1$ this equals $2g-2$: it is $2$ when $g=2$, but is larger than $\dim |H|_{\sm}=g$ for every $g\ge3$. Thus a sufficiently general morphism of these ranks would have isolated ramification in genus $2$ and no ramification at all in higher genus. This numerical heuristic might explain the contrast between the $171$ ramification points in degree $2$ (Theorem~\ref{thm:A}) and the absence of ramification for the very general quartic (Theorem~\ref{thm:B}). The analogous numerology for $d\mu_n$ on $|nH|_{\sm}$ is less predictive: it gives expected dimension $2-g_n=1-n^2H^2/2$, which is negative whenever $g_n\ge3$, in particular for every $n\ge2$, despite the non-primitive ramification examples constructed above (Sections~\ref{sec:stability-degree2} and~\ref{sec:stability-degree4}). One could thus hope that for degree $6$ onwards, the primitive class contains no smooth ramification points, but the story is more subtle for multiples.
\end{remark}

\begin{remark}[Quasi-finiteness and the cases treated in this paper]\label{rem:jumping-quasifinite}
If the moduli map is quasi-finite, as it is for a primitive, ample and base-point-free line bundle on a Picard-rank-one $K3$ surface by \cite[Theorem~A]{CG}, then \Cref{prop:jumping-codimension} gives
\[
\codim_{|H|_{\sm}}S_k(H_{\sm})\ge k.
\]
Thus the $171$ reduced ramification points in degree $2$ lie in $S_1(H_{\sm})$, while $S_2(H_{\sm})=\varnothing$ (\Cref{thm:very-general_h2}); for a very general quartic, even $S_1(H_{\sm})$ is empty (\Cref{cor:quartic-unramified}). By contrast, the non-primitive constructions of \Cref{sec:stability-degree2} and \Cref{sec:stability-degree4} produce positive-dimensional ramification: a locus of dimension at least $15$ in $|6H|_{\sm}$ in degree $2$, and, for a quartic, a family parametrised by the $8$-dimensional Grassmannian $\operatorname{Gr}(2,6)$. In both cases $h^0(C,T_S|_C)=1$, so the general constructed curve lies in $S_1\setminus S_2$. In particular, we expect that the expected codimension bounds of the proposition above are weak.
\end{remark}

\begin{remark}[Relation with stability and Harder--Narasimhan strata]\label{rem:jumping-stability}
The filtration by the $S_k(H)$ could be compared with the stratification of $|H|_{\sm}$ by the Harder--Narasimhan type of $T_S|_C$. Since
$\det(T_S|_C)\cong\Oc_C$,
a nonzero section of $T_S|_C$ determines, after saturation, an inclusion
\[
\begin{tikzcd}
\Oc_C(D) \arrow[r, hook] & T_S|_C
\end{tikzcd}
\]
for an effective divisor $D$. If $D>0$, this is a destabilising subbundle. If $D=0$, the bundle can remain semistable but cannot be stable. This separates two geometrically different parts of $S_1(H)$: sections with zeros, which force non-semistability, and nowhere-vanishing sections, which only force failure of stability.

The converse fails in general, as Section~5 illustrates: a destabilising line bundle of positive degree need not have a section. One could refine the jumping loci by fixing the degree and Brill--Noether behaviour of the saturated line subbundle,
\[
S_{k,e}(H):=\{[C]\in |H|_{\sm}:\text{there is }\Oc_C(D)\subset T_S|_C,\ 
                         \deg D=e,\ h^0(C,T_S|_C)\ge k\}.
\]
A relative Harder--Narasimhan construction should make the closures of these loci algebraic after passing to an appropriate relative moduli space of bundles. Comparing their images with $S_k(H)$ could explain which instability phenomena actually cause ramification and which remain invisible to the moduli map, but we do not pursue this here.
\end{remark}

\appendix

\section{Computational verification of the reduced jumping scheme}
\label{app:jumping}\label{app:reduced-jumping}

The following computation supplies the reduced jumping locus used
in \Cref{prop:reduced-example_h2}. The computation shows that the entire projective jumping
locus is contained in the first affine chart and has length $171$.  The
Jacobian criterion then proves that it is reduced. Finally, the binary resultant of the restriction of $F_1$ to the universal line is the degree-$30$
equation for tangency, and its intersection with the jumping ideal is empty.

\begingroup
\scriptsize
\begin{verbatim}
kk = QQ;
R = kk[X,Y,Z];
A = kk[u,v];
B = A[x,y];
needsPackage "Resultants";

F1 = (X^6+Y^6+Z^6+X^5*Y+Y^5*Z+Z^5*X
      +2*X^3*Y^2*Z+3*X^2*Y*Z^3+5*X*Y^3*Z^2);

jumpingIdeal = (F,chart) -> (
    d0 := diff(X,F); d1 := diff(Y,F); d2 := diff(Z,F);
    phi := if chart == 1 then map(B,R,{x,y,u*x+v*y})
           else if chart == 2 then map(B,R,{u*x+v*y,x,y})
           else map(B,R,{x,u*x+v*y,y});
    rows := {x*phi(d0),y*phi(d0),x*phi(d1),
             y*phi(d1),x*phi(d2),y*phi(d2)};
    (mons,coeffsB) := coefficients(matrix{rows},Variables=>{x,y});
    coeffsA := substitute(coeffsB,A);
    minors(6,coeffsA)
    );

-- Smoothness of the sextic.
JF1 = ideal(diff(X,F1),diff(Y,F1),diff(Z,F1));
print("dimension of affine Jacobian scheme: "|toString dim JF1);

-- Jumping ideals on the three standard affine charts.
I1 = jumpingIdeal(F1,1);
I2 = jumpingIdeal(F1,2);
I3 = jumpingIdeal(F1,3);
print("patch 1 (dimension, degree): "|toString(dim I1,degree I1));
print("patch 2 (dimension, degree): "|toString(dim I2,degree I2));
print("patch 3 (dimension, degree): "|toString(dim I3,degree I3));

-- The complement of patch 1 is v=0 in patches 2 and 3.
print("boundary in patch 2 is empty: "|
      toString(dim(I2+ideal(v)) == -1));
print("boundary in patch 3 is empty: "|
      toString(dim(I3+ideal(v)) == -1));

-- Jacobian criterion for the zero-dimensional scheme in patch 1.
SingI1 = I1 + minors(2,jacobian gens I1);
print("jumping scheme is reduced: "|toString(dim SingI1 == -1));

-- The degree-30 equation of the dual curve on patch 1.
phi1 = map(B,R,{x,y,u*x+v*y});
Q1 = phi1(F1);
Delta1 = resultant {diff(x,Q1),diff(y,Q1)};
print("degree of tangency equation: "|toString first degree Delta1);
print("no jumping line is tangent: "|
      toString(dim(I1+ideal(Delta1)) == -1));
\end{verbatim}
\endgroup

\end{document}